\documentclass[12pt, letterpaper,final]{amsart}
\usepackage{amsfonts,amsmath, amsthm, amssymb, latexsym, epsfig}
\usepackage[english]{babel}
\usepackage[utf8]{inputenc}
\usepackage[all]{xy}
\usepackage{xspace}
\usepackage{comment}
\usepackage{setspace}
\usepackage{enumerate}
\usepackage{stmaryrd}
\usepackage{xcolor}
\usepackage{hyperref}
\usepackage[mathscr]{eucal}
\usepackage[notcite,notref]{showkeys}

\numberwithin{equation}{section}

\allowdisplaybreaks

	\newcommand{\ZZ}{\mathbb{Z}}

	\newcommand{\NN}{\mathbb{N}}

	\newtheorem{thm}{Theorem}[section]
	
	\newtheorem{lem}[thm]{Lemma}
	\newtheorem{prp}[thm]{Proposition}

	\newtheorem*{recast}{Recasting the Hazrat Conjecture}

	\newtheorem*{hazrat}{Hazrat's Graded Morita Equivalence Conjecture}

	\theoremstyle{definition}
	\newtheorem{dfn}[thm]{Definition}
	\newtheorem{ntn}[thm]{Notation}

	\theoremstyle{remark}
	\newtheorem{rmk}[thm]{Remark}
	\newtheorem{exm}[thm]{Example}

\begin{document}
	\title[Recasting the Hazrat Conjecture]{Recasting the Hazrat Conjecture:  \\ 
	\vspace{.05in}
	 Relating Shift Equivalence to \\ Graded Morita Equivalence
}
	\author{Gene Abrams}
	\address{Department of Mathematics \\ University of Colorado \\ 1420 Austin Bluffs Parkway \\  Colorado Springs, CO 80918 U.S.A.}
	\email{abrams@math.uccs.edu}
	\author{Efren Ruiz}
        \address{Department of Mathematics\\University of Hawaii,
Hilo\\200 W. Kawili St.\\
Hilo, Hawaii\\
96720-4091 U.S.A.}
        \email{ruize@hawaii.edu}
        
        \author{Mark Tomforde}
	\address{Department of Mathematics \\ University of Colorado \\ 1420 Austin Bluffs Parkway \\  Colorado Springs, CO 80918 U.S.A.}
	\email{mtomford@uccs.edu}
        
        \date{\today}
	
\begin{abstract} 
Let $E$ and $F$ be finite graphs with no sinks, and $k$ any field.   We show that shift equivalence of the adjacency matrices $A_E$ and $A_F$, together  with an additional compatibility condition,  implies that the Leavitt path algebras $L_k(E)$ and $L_k(F)$ are graded Morita equivalent.   Along the way, we build a new type of $L_k(E)$--$L_k(F)$-bimodule (a {\it bridging bimodule}),  which we use to establish the 
graded equivalence.
\end{abstract}

\thanks{This work was supported by grants from the Simons Foundation (\#527708 to Mark Tomforde and \#567380 to Efren Ruiz)}

\dedicatory{Dedicated to the memory of Professor Iain Raeburn.  \\  The collaboration of the  coauthors on this article   is a direct result of Iain's profound influence \\ on the field in general, and on the three of us individually.  He will be deeply missed.}

\maketitle

\section{Introduction}

A number of longstanding, as-yet-unresolved questions in the study of Leavitt path algebras have at their heart the search for  ``easily computable" data about the underlying graphs that would yield ring-theoretic connections between the corresponding algebras.  Perhaps the  most intensely investigated  of these questions is the {\it Hazrat Graded Morita Equivalence Conjecture}, which we describe here. 

Two square matrices $A$ and $B$ (not necessarily of the same size) with nonnegative integer entries are called  \emph{shift equivalent} provided that there exists a positive integer $n$, and rectangular matrices $R$ and $S$ with nonnegative integer entries, for which 
$A^n = RS ,  B^n = SR,  A R = R B,   \mbox{ and }  B S = S A.$   
Two group-graded rings $T_1$ and $T_2$ are called {\it graded Morita equivalent} in case there is an equivalence functor (compatible with the gradings)  between the categories   $\mathsf{Gr}$-$T_1$ and $\mathsf{Gr}$-$T_2$ of  graded  right modules over $T_1$ and $T_2$, respectively.        

 Let $k$ be any field, and $E$ and $F$ any directed graphs.     The Leavitt path algebras $L_k(E)$ and $L_k(F)$ are  $\mathbb{Z}$-graded in a natural way.  It is known that for two finite sink-free graphs $E$ and $F$,  graded Morita equivalence of the Leavitt path algebras $L_k(E)$ and $L_k(F)$ implies that the adjacency matrices  of $E$ and $F$ are shift equivalent.   Hazrat's Conjecture (see \cite[Remark 16]{RoozbehDyn})  posits that the converse of this implication holds as well.    Here is an equivalent  formulation of the Conjecture.

\begin{hazrat}
Let $k$ be any field.  Let $E$ and $F$ be finite graphs with no sinks.   Then the following are equivalent.
\begin{quote}
\begin{enumerate}
\item[$\mathrm{(GrME)}$] The Leavitt path algebras $L_k(E)$ and $L_k(F)$ are graded Morita equivalent.

\item[$\mathrm{(GrK)}$] There is an order preserving $\mathbb{Z}[x,x^{-1}]$-module isomorphism from $K_0^{\mathrm{gr}}(L_k(E))$ to $K_0^{\mathrm{gr}}(L_k(F))$.

\item[$\mathrm{(SE)}$] The adjacency matrices of $E$ and $F$ are shift equivalent.
\end{enumerate}
\end{quote}
\end{hazrat}

The implication  ${\rm (GrME)} \implies {\rm (SE)}$ of the Conjecture is established in \cite{RoozbehDyn}.    Hazrat achieves this by first proving the equivalence of statements ${\rm (GrK)}$ and  $ {\rm (SE)}$ in  \cite[Corollary 12]{RoozbehDyn}, and then showing that    ${\rm (GrME)} \implies {\rm (GrK)}$ as part of the proof of \cite[Proposition 15(3)]{RoozbehDyn}.    The proofs of these results use some significant machinery from symbolic dynamics (e.g., Krieger's dimension groups).     The implication in the Conjecture which has yet to be established (or disproved) is      ${\rm (SE)} \implies {\rm (GrME)}$.     (``One thinks that the converse of the statement ${\rm (GrME)} \Rightarrow {\rm (SE)}$ is also valid."  See \cite[Remark 16]{RoozbehDyn}.)    Additional information about the Conjecture, and variants thereof, can be found in  \cite{AraPardo} and \cite{RoozbehAnnalen}. 

The main contribution we make  in this article towards resolving   Hazrat's Conjecture consists of two steps.  First, we find  an equivalent formulation of Condition ${\rm (SE)}$ (which we denote by ${\rm (MSE)}$);  Condition ${\rm (MSE)}$ is an assertion about the existence of isomorphisms between various bimodules, and is appropriately  referred to as {\it module shift equivalence}.   We then subsequently prove an implication of the form 
$${\rm (MSE)} \mbox{ coupled with a compatibility condition on the isomorphisms }    \implies  \  {\rm (GrME)}.$$
  
The article is organized as follows.  In Section \ref{PrelimSection} we set notation and remind the reader of various germane ring-theoretic and graph-theoretic ideas.     
 In Section \ref{3iff3primeSection} we establish our first main result (Theorem \ref{thm:shift-equivalence-module}), in which we show the equivalence of the aforementioned Conditions ${\rm (SE)}$ and ${\rm (MSE)}$.    The significance of Theorem~\ref{thm:shift-equivalence-module} is that it allows one to replace Condition~{\rm (SE)} of Hazrat's Conjecture (which is a condition on adjacency matrices of the graphs) with Condition~${\rm (MSE)}$ (which is a condition on isomorphisms between tensor products of bimodules associated to the graphs). This is useful because Condition~${\rm (MSE)}$ allows for a natural way to append an additional hypothesis,   thereby allowing us to establish a variant of the one not-yet-resolved implication in Hazrat's Conjecture.  
 
We give here an intuitive description of the motivation which led to the idea of replacing Condition~{\rm (SE)} with   Condition~${\rm (MSE)}$.  For shift equivalent matrices $A$ and $B$ we have  $A(RS) = AA^n$.  As well, we have 
$ A (RS) = (AR) S = (RB) S = R (BS) = R (SA) = (RS) A = A^n A$, so that $A^nA = A(RS) = AA^n$.  The point to be made is that we may interpret $A(RS)$ both as $AA^n$ and as $A^nA$;  obviously these two interpretations lead to the same matrix.  In Section \ref{3iff3primeSection}  we show how to associate any rectangular matrix $M$ having nonnegative integer entries with a bimodule $X(M)$ (over appropriate algebras).   In the bimodule setting, these  computations involving products of matrices will lead to an analogous statement regarding isomorphisms of bimodules:    
$$X(A)^{\otimes n} \otimes X(A)  \ \cong \ X(A) \otimes (X(R) \otimes X(S)) \ \cong \  X(A) \otimes X(A)^{\otimes n}.$$
 Of course there is only one way for two matrices to be equal.  In contrast,   bimodules may be isomorphic via more than one isomorphism.  Recast, the diagram of {\it matrix  equalities}  (i.e., Condition ${\rm (SE)}$)
$$
\xymatrix{
A (RS) \ar@{=}[r] \ar@{=}[d] & (RS)A \ar@{=}[d] \\
A A^n \ar@{=}[r] & A^n A
}
$$
will be shown to yield   a   diagram of  {\it isomorphisms of bimodules} (i.e., Condition ${\rm (MSE)}$)
$$
\xymatrix{
X(A) \otimes (X(R) \otimes X(S))  \ar[d] 
\ar[r]
   & (X(R) \otimes X(S)) \otimes X(A)  \ar[d]  \ar[l] \\
 \ar[u] X(A)  \otimes (X(A)^{\otimes n}) \ar[r]   & \ar[l] (X(A)^{\otimes n}) \otimes X(A) \ar[u]
}
$$
The point then is that we may add a natural  condition on top of the existence of this diagram of bimodule isomorphisms:  specifically, we may hypothesize  that there exists such a diagram of isomorphisms that commutes.

Suppose $E$ and $F$ are finite graphs with no sinks whose adjacency matrices are shift equivalent.   In Section~\ref{sec-const-module} we use Condition ${\rm (MSE)}$ to construct an algebraic analogue of Ery\"uzl\"u's $C^*$-correspondence \cite{Eryuzlu}.  Specifically, we  build an  $L_k(E)$--$L_k(F)$-bimodule $Y$, which we call a \emph{bridging bimodule}.   While the right $L_k(F)$-module action on $Y$ will be the obvious one, the construction of a left $L_k(E)$-module action on $Y$ will require some significant effort, which we complete in Theorem \ref{thm:bimodule-action}.  We then establish basic properties of bridging bimodules throughout Section \ref{bimodule-key-properties-sec}.     

In the final section (Section \ref{Theorem-B-sec}) we use the previous work to establish our second main result, Theorem~\ref{thm:main2}.   In it, we show that Condition ${\rm (SE)}$ in Hazrat's Conjecture (shift equivalence of the adjacency matrices), when reformulated as Condition ${\rm (MSE)}$ and paired with the aforementioned commutativity condition, implies Condition ${\rm (GrME)}$ of the Conjecture (graded equivalence of the Leavitt path algebras).  The tensor product functor  induced  by the  bridging bimodule $Y$ provides the graded equivalence between the categories of graded modules over the two Leavitt path algebras.  

 We conclude the article with two subsections.  In Subsection \ref{SubsectionSSE} we show how Theorem \ref{thm:main2} yields that {\it strong} shift equivalence implies graded equivalence of the Leavitt path algebras, thereby slightly generalizing a known result, while utilizing  a vastly different approach for the proof.  Finally, in Subsection \ref{SubsectionHazratConj}, we discuss the current ``state of affairs" regarding the resolution of the Hazrat Conjecture.   We point out various conditions which imply, and which are implied by, the statement that two Leavitt path algebras $L_k(E)$ and $L_k(F)$ are graded Morita equivalent (where $E$ and $F$ are finite graphs with no sinks).   We then remark on how various  consequences would result from the establishment of  any one of these implications.   
 
As it turns out, a number of the key constructions in our work can be viewed in a more general categorical / functorial framework, see for instance Remarks \ref{polymorphism-category-rem}, \ref{poly-and-matrix}, \ref{BiModremark},  and \ref{categoricalperspectiveremark}.   While we will not  pursue this point of view in the current article, we  believe that these observations may provide a starting point for future investigations.

\section{Preliminaries}\label{PrelimSection}

 We set some notation.  We denote the set of nonnegative integers by $\mathbb{Z}_{\geq 0}$, and the set of positive integers by $\mathbb{N}$.     We write functors and functions   (including left-module homomorphisms) on the left, so that $g \circ f$ will always mean ``first $f$, then $g$".  The expression ${\rm id}$ will denote an identity morphism, interpreted in the appropriate context.  
 
 A (directed) graph $E = (E^0, E^1, s,r)$ consists of a {\it vertex set} $E^0$, an {\it edge set} $E^1$, and {\it source} and {\it range} functions $s, r: E^1 \rightarrow E^0$.  For $v\in E^0$, the set of edges $\{ e\in E^1 \ | \ s(e)=v\}$ is denoted by $s^{-1}(v)$, and the set of edges $\{ e\in E^1 \ | \ r(e)=v\}$ by $r^{-1}(v)$. The graph $E$ is called {\it finite} in case both $E^0$ and $E^1$ are finite sets.  A vertex $v$ is a {\it sink vertex} (or simply a {\it sink}) in case $s^{-1}(v) = \emptyset$;  i.e., in case $v$ is not the source vertex of any edge in $E$.   
A {\it path} $\alpha$ in $E$ is a sequence $e_1 e_2 \cdots e_n$ of edges in $E$ for which $r(e_i) = s(e_{i+1})$ for all $1 \leq i \leq n-1$.  
We say that such $\alpha$ has {\it length} $n$ (denoted by $\ell(\alpha) = n$).  We write $s(\alpha) = s(e_1)$ and $r(\alpha) = r(e_n)$.  We view each vertex $v \in E^0$ as a path of length $0$, and denote $v = s(v) = r(v)$.    For $n\geq 0$ we let $E^n$ denote the paths of length $n$ in $E$, and define ${\rm Path}(E) = \bigcup_{n\geq 0}E^n$.    For $v\in E^0$ and $n\geq 0$  we define $vE^n := \{ \alpha \in E^n : s(\alpha) = v \}$.

\smallskip

     Let $k$ be a field, and let $E = (E^0, E^1, s,r)$ be a directed  graph with vertex set $E^0$ and edge set $E^1$.   The {\em Leavitt path algebra} $L_k(E)$ {\em of $E$ with coefficients in $k$} is  the $k$-algebra generated by a set $\{v\mid v\in E^0\}$, together with a set of symbols $\{e,e^*\mid e\in E^1\}$, that satisfy the following relations:

\smallskip

(V)   \ \ \  \ $vu = \delta_{v,u}v$ for all $v,u\in E^0$, \  

(E1) \ \ \ $s(e)e=er(e)=e$ for all $e\in E^1$,

(E2) \ \ \ $r(e)e^*=e^*s(e)=e^*$ for all $e\in E^1$,

(CK1) \ $e^*e'=\delta _{e,e'}r(e)$ for all $e,e'\in E^1$, and

(CK2)  \ 
  $v = \sum_{  e \in vE^1   } ee^\ast$ for every   $v\in E^0$ for which   $0 < |s^{-1}(v)| < \infty$. 

\noindent 
An alternate description of $L_k(E)$ may be given as follows.  For any graph $E$ let $\widehat{E}$ denote the ``double graph" of $E$, obtained by adding to $E$ an edge $e^*$ in a reversed direction for each edge $e\in E^1$.   Then $L_k(E)$ is the usual path $k$-algebra $k\widehat{E}$, modulo the ideal generated by the relations (CK1) and (CK2).

\smallskip

It is easy to show that $L_k(E)$ is unital if and only if $E^0$ is finite; in this case, $1_{L_k(E)} = \sum_{v\in E^0}v$.    Every nonzero element of $L_k(E)$ may be written as $\sum_{i=1}^n k_i \alpha_i \beta_i^*$, where $k_i$ is a nonzero element of $k$, and each of the $\alpha_i$ and $\beta_i$ are paths in $E$ with $r(\alpha_i) = r(\beta_i)$.  The  map $kE \rightarrow L_k(E)$ given by the $k$-linear extension of $\alpha \mapsto \alpha$ (for $\alpha \in {\rm Path}(E)$) is an injection of $k$-algebras by \cite[Corollary 1.5.12]{TheBook}.   

If $A$ is any unital $k$-algebra and $E$ is any  graph,  then a {\it Cuntz Krieger} $E$-{\it family in} $A$ is a subset $\{P_v \ | \ v\in E^0 \} \cup \{S_e,  S_{e^*} \ | \ e\in E^1\}$ of $A$ which satisfies  the five analogous relations (V), (E1), (E2), (CK1), and (CK2) given above.   By the universal property of the free associative $k$-algebra on the symbols $\{v \ | \ v\in E^0 \} \sqcup \{e, e^* \ | \ e\in E^1\}$, if $\{P_v \ | \ v\in E^0 \} \cup \{S_e , S_{e^*} \ | \ e\in E^1\} $ is a Cuntz Krieger $E$-family in $A$ then there exists a unique $k$-algebra homomorphism $\varphi: L_k(E) \to A$ for which $\varphi(v) = P_v$, $\varphi(e) = S_e$, and $\varphi(e^*) = S_{e^*}$ for all $v\in E^0$ and $e\in E^1$.   (See e.g., \cite[Remark 1.2.5]{TheBook}.)

For any graph $E$ and field $k$ there is a $\mathbb{Z}$-grading on $L_k(E)$ defined as follows.  For each $n\in \mathbb{Z}$, the $n$th graded component $(L_k(E))_n$ is the $k$-subspace of $L_k(E)$ generated by elements of the form $\alpha \beta^*$ where $\alpha$ and $\beta$ are paths in $E$ for which $\ell(\alpha) - \ell(\beta) = n$.   We call this grading the {\it standard} $\mathbb{Z}$-grading on $L_k(E)$.   

For additional information about Leavitt path algebras, see e.g. \cite{TheBook}.  

For a ring $R$ graded by the abelian group $(\mathcal{G},+)$, a  $\mathcal{G}$-{\it graded right (resp., left) } $R$-module is a right (resp., left) $R$-module $M$ for which  $M = \oplus_{g\in \mathcal{G}}M_g$ as abelian groups, and for which $m_g r_h \in M_{g+h}$  (resp., $r_h m_g\in M_{g+h}$)  for all $g,h\in \mathcal{G}$.  When both $R$ and $S$ are  $\mathcal{G}$-graded rings, and ${}_S M _R$ is an $S$--$R$-bimodule, then $M$ is a $\mathcal{G}$-{\it graded bimodule} in case $M = \oplus_{g\in \mathcal{G}}M_g$ as abelian groups in such a way that this decomposition simultaneously makes  $M$ both a $\mathcal{G}$-graded left $S$-module and a $\mathcal{G}$-graded right $R$-module.     Throughout the article, the word ``graded" will always mean ``$\mathbb{Z}$-graded", unless otherwise indicated.

As described in the introduction, we wish to characterize  shift equivalence of the adjacency matrices of graphs stated in Condition~{\rm (SE)} of Hazrat's Conjecture in terms of isomorphisms of  tensor products of  bimodules associated to those graphs.  Part of our reason for doing so is that in addition to considering \emph{whether} those tensor product bimodules are isomorphic, we will also be in position to  ask \emph{how} the bimodules are isomorphic.  Specifically,  we will consider whether the  isomorphisms satisfy certain  (quite natural) commutativity conditions.

We will see that multiplication of adjacency matrices corresponds to  tensor product of the bimodules. By associativity of matrix multiplication we have  $(AB)C = A(BC)$;  and when we are only concerned with equality we simply write $ABC$ for this common value.  We also have associativity for tensor products:  if $X$ is an $R$--$S$-bimodule, $Y$ is an $S$--$T$-bimodule, and $Z$ is a $T$--$U$-bimodule, then $(X \otimes_S Y) \otimes_T Z \cong X \otimes_S (Y \otimes_T Z)$.  However, since we will need to explicitly invoke these isomorphisms, whenever we use this associativity of tensor products  we will need to keep track of the isomorphism implementing the associativity.  To do so, we  use the following notation.

\begin{ntn}[Associativity Isomorphism for Bimodules] \label{associativity-iso-notation}
If $X$ is an $R$--$S$-bimdoule, $Y$ is an $S$--$T$-bimodule, and $Z$ is a $T$--$U$-bimodule, we let
$$\alpha_{X,Y,Z} : (X \otimes_S Y) \otimes_T Z \to X \otimes_S (Y \otimes_T Z)$$
denote the unique $R$--$U$-bimodule isomorphism for which $$\alpha_{X,Y,Z} ((x \otimes y) \otimes z) = x \otimes (y \otimes z)$$
for all $x\in X$, $y\in Y$, and $z\in Z$.  
\end{ntn}

\section{Characterizing Shift Equivalence of Graphs \\ in terms of Shift Equivalence of Modules}\label{3iff3primeSection}

Our goal in this section is to prove Theorem~\ref{thm:shift-equivalence-module}, which is the first of our two main results.  Theorem~\ref{thm:shift-equivalence-module} establishes that shift equivalence of the adjacency matrices of two graphs is equivalent to a condition (that we call Condition~${\rm (MSE)}$)  on isomorphisms between  bimodules over algebras associated to the graphs.  Thus Condition~${\rm (MSE)}$ provides a notion of ``shift equivalence of modules"; rephrased,  Theorem~\ref{thm:shift-equivalence-module} asserts that shift equivalence of graphs is equivalent to shift equivalence of the associated modules.

\subsection{Polymorphisms between Sets}

Adjacency matrices of graphs are always square.  To deal with the rectangular matrices that appear in the definition of shift equivalence, we are led naturally to consider bipartite graphs consisting of edges going from a set of one cardinality to a set of another cardinality.  We develop the notion of a polymorphism to make this idea precise.

\begin{dfn}
A \emph{polymorphism} from a set $V$ to a set $W$ is a $5$-tuple $E := (V, W, E^1, r, s)$  consisting of sets $V$, $W$, and $E^1$ with functions $s: E^1 \to V$ and $r : E^1 \to W$.  When the maps $r$ and $s$ are understood, we shall commonly use the shorthand ${}_VE_W$ in analogy with the notation for bimodules.  We say that a polymorphism is {\it finite} when $V$, $W$, and $E^1$ are all finite sets. 

Note  that a directed graph is the special case of a polymorphism with $V = W$.  For fixed sets $V$ and $W$, two polymorphisms $E := (V, W, E^1, r_E, s_E)$ and $F := (V, W, F^1, r_F, s_F)$ are \emph{isomorphic} when there exists a bijection $\phi : E^1 \to F^1$ with $s_F ( \phi(e)) = \phi(s_E(e))$ and $r_F ( \phi(e)) = \phi(r_E(e))$ for all $e \in E^1$.  (Equivalently, an isomorphism from $E$ to $F$ is a function from edges of $E$ to edges of $F$ with the property that for all $v \in V$ and $w \in W$, the function maps the set of edges in $E$ from $v$ to $w$ bijectively onto the set of edges in $F$ from $v$ to $w$.  Consequently, if $E$ and $F$ are two polymorphisms from $V$ to $W$, then $E$ is isomorphic to $F$ if and only if for all $v \in V$ and $w \in W$ the polymorphisms $E$ and $F$ have the same number of edges from $v$ to $w$.) 
\end{dfn}

\begin{dfn} \label{adjacency-matrix-def}
For any polymorphism  $E := (V, W, E^1, r, s)$, the \emph{adjacency matrix} of $E$, denoted $A_E$, is the $V \times W$ matrix with entries defined by setting
$$A_E(v,w) := \# \{ e \in E^1 : s(e) = v \text{ and } r(e) = w \} $$
for all $v \in V$ and $w \in W$.
Note that the adjacency matrix of a finite polymorphism is a rectangular matrix with entries in $\mathbb{Z}_{\geq 0}$.   Conversely, any finite matrix indexed by $V \times W$ and with entries in $\mathbb{Z}_{\geq 0}$ determines a finite polymorphism by simply drawing $A(v,w)$ edges from $v$ to $w$ for all $v \in V$ and $w \in W$.  Moreover, if $E$ and $F$ are polymorphisms from $V$ to $W$, then $E$ is isomorphic to $F$ if and only if $A_E = A_F$.
\end{dfn}

\begin{dfn}[Product of Polymorphisms]
If $E := (U, V, E^1, r_E, s_E)$ is a polymorphism from $U$ to $V$ and $F := (V, W, F^1, r_F, s_F)$ is a polymorphism from $V$ to $W$, we define $$E \times F := (U, W, (E \times F)^1, r_{E \times F}, s_{E \times F})$$ to be the polymorphism from $U$ to $W$ with
$$(E \times F)^1 := \{ (e,f) : e \in E^1, f \in F^1, \text{ and } r_E(e) = s_F(f) \}$$
and with $r_{E \times F} (e,f) = r_F(f)$ and $s_{E \times F} (e,f) = s_E(e)$.
One can verify that the adjacency matrices of these polymorphisms satisfy $A_{E \times F} = A_E A_F$.  (Note:  It would probably be more descriptive to use the notation $E \times _V F$ in place of $E \times F$, thereby emphasizing the requirement that the codomain of $E$ coincides with the domain of $F$.  However, this notation quickly becomes cumbersome, so we find it convenient to suppress the $V$ on the product symbol.)

If $E := (E^0, E^1, r, s)$ is a graph, for any $n \in \mathbb{N}$ we denote the $n$-fold product of $E$ with itself by $E^{\times n} := E \times \cdots \times E$.  Note the special cases that $E^{\times 1} = E$ and $E^{\times 2} = E \times E$. Also observe that $A_{E^{\times n}} = A_E^n$; that is, the adjacency matrix of $E^{\times n}$ is the $n$th power of the adjacency matrix of $E$.
\end{dfn}

\begin{rmk} \label{polymorphism-category-rem}
We can view (the isomorphism class of) a polymorphism $E := (V, W, E^1, r_E, s_E)$ as a morphism from $V$ to $W$ with the product of polymorphisms providing a composition of morphisms.  More precisely, we may form a category $\mathsf{Poly}$ whose objects are finite sets and whose morphisms are isomorphism classes of finite polymorphisms.  If $E$ is a polymorphism from $U$ to $V$ and if $F$ is a polymorphism from $V$ to $W$, with $[E]$ and $[F]$ denoting the respective isomorphism classes of these polymorphisms, composition of morphisms in this category is defined as
$$ [F] \circ [E] := [ E \times F].$$
One can verify that the axioms of a category are satisfied, and that for any set $V$ the identity morphism  on $V$ is the isomorphism class of the polymorphism $\operatorname{Id}_V := (V, V, V, \operatorname{id}_V, \operatorname{id}_V)$.  Note that with this categorical viewpoint, a finite graph $E := (E^0, E^1, r, s)$ determines an endomorphism on the set $E^0$. 
\end{rmk}

\begin{rmk}\label{poly-and-matrix}
We may define a category $\mathsf{Mat}(\mathbb{Z}_{\geq 0})$ whose objects are elements of $\mathbb{Z}_{\geq 0}$, and for two objects $m,n \in \mathbb{Z}_{\geq 0}$ the morphisms from $m$ to $n$ are the $m \times n$ matrices with entries in $\mathbb{Z}_{\geq 0}$, with composition given by matrix multiplication; i.e., if $A$ goes from $m$ to $n$, and if $B$ goes from $n$ to $l$, then $B \circ A := AB$.  It is straightforward to verify that $\mathsf{Poly}$ is equivalent to the category $\mathsf{Mat}(\mathbb{Z}_{\geq 0})$, and an equivalence (i.e., a full, faithful, and dense functor) may be obtained as follows:  on objects the functor assigns each set to its cardinality  $V \mapsto |V|$, while on morphisms the functor assigns each isomorphism class of a polymorphism to the adjacency matrix of a representative $[E] \mapsto A_E$.
\end{rmk}

\begin{dfn}[The Bimodule of a Polymorphism] \label{graph-bimodule-def}
Let $k$ be a field.  For a set $V$, we let $kV$ denote the $k$-vector space with basis $V$.  We may also make $kV$ into a $k$-algebra by defining $vw = \delta_{v,w} v$ for $v, w \in V$.
Observe that the $k$-algebra $kV$ is simply the direct sum $\bigoplus_{v \in V} k$ with one copy of $k$ for each element of $V$.   

Clearly $kV$ is unital if and only if $V$ is a finite set.   But in any event, $kV$ contains a set of local units, and in this setting the notion of a $kV$-{\it module}  is well-understood.

If $E := (V, W, E^1, r, s)$ is a polymorphism from $V$ to $W$, we make the vector space $kE^1$ into a $kV$--$kW$-bimodule, called the \emph{polymorphism bimodule}, with left and right module actions defined by the $k$-linear extensions of
$$
v \cdot e := \begin{cases} e & \text{ if $s(e) = v$} \\ 0 & \text{ if $s(e) \neq v$} \end{cases} \quad \qquad \text{ and } \quad \qquad e \cdot w := \begin{cases} e & \text{ if $r(e) = w$} \\ 0 & \text{ if $r(e) \neq w$} \end{cases}
$$
for $v \in V$, $e \in E^1$, and $w \in W$.  
In particular, when $E = (E^1, E^0, r, s)$ is a graph, then $kE^1$ is a $kE^0$--$kE^0$-bimodule.   The aforementioned embedding of $kE^1$ into $L_k(E)$   allows us to view $kE^1$ as a $kE^0$-submodule of $L_k(E)$.   
\end{dfn}

 The following result shows that the bimodules between two $k$-algebras of the form $kV$   are precisely the polymorphism bimodules.

\begin{prp} \label{bimodule-include-prop}
Let $k$ be a field, let $V$ and $W$ be sets, and let $M$ be a $kV$--$kW$-bimodule.  Then there exists a polymorphism $E := (V, W, E^1, r, s)$ such that ${kE^1} \cong M$ as a $kV$--$kW$-bimodule.  Moreover, such a  polymorphism $E$ is unique up to isomorphism.
\end{prp}

\begin{proof}
Let $v \in V$ and $w \in W$.  Then $vMw$ is a $k$-submodule of $M$, and hence a $k$-vector space.  For each $v \in V$ and $w \in W$ let $\{ e_i^{v,w} : i \in I_{v,w} \}$ be a basis for $vMw$ indexed by the set $I_{v,w}$.  Define 
$$E^1 := \{  e_i^{v,w} : \text{$v \in V$, $w \in W$, and $i \in I_{v,w}$} \}$$
and let $E := (V, W, E^1, r, s)$ be the polymorphism with $s(e_i^{v,w}) = v$ and $r(e_i^{v,w}) = w$ for all $e_i^{v,w} \in E^1$.  Recall that $kE^1$ is the $kV$--$kW$-bimodule with
$$v \cdot e \cdot w := \begin{cases} e & \text{ if $s(e)=v$ and $r(e)=w$} \\ 
0 & \text{ otherwise.}\end{cases}$$
Since $E^1$ is a $k$-basis for $kE^1$, it is straightforward to verify that the map
$$E^1 \ni e_i^{v,w}  \mapsto e_i^{v,w} \in M$$
extends to a $kV$--$kW$-bimodule isomorphism from ${kE^1}$ to $M$, and hence ${kE^1} \cong M$ as $kV$--$kW$-bimodules.

For uniqueness of $E$, suppose $E := (V, W, E^1, r_E, s_E)$ and $F := (V, W, F^1, r_F, s_F)$ are polymorphisms from $V$ to $W$ with ${kE^1} \cong {kF^1}$ as $kV$--$kW$-bimodules.  Let $\phi : {kE^1} \to {kF^1}$ be a $kV$--$kW$-bimodule isomorphism.  Then for all $v \in V$ and $w \in W$ we have
$$\phi (v{kE^1}w) = v \phi({kE^1})w = v{kF^1}w,$$
and hence $\phi$ restricts to an isomorphism $\phi : v{kE^1}w \to v{kF^1}w$.
Since $v{kE^1}w = \operatorname{span}_k \{ e \in E^1 : s_E(e) = v \text{ and } r_E(e) = w \}$, we conclude $\{ e \in E^1 : s_E(e) = v \text{ and } r_E(e) = w \}$ is a basis for $v{kE^1}w$.  Likewise, we have $v{kF^1}w = \operatorname{span}_k \{ e \in F^1 : s_F(e) = v \text{ and } r_F(e) = w \}$, so that $\{ e \in F^1 : s_F(e) = v \text{ and } r_F(e) = w \}$ is a basis for $v{kF^1}w$.  Since $\phi : {kE^1} \to {kF^1}$ is an isomorphism, the cardinalities of these bases are the same.  Hence for all $v \in V$ and $w \in W$, the cardinality of the set of edges in $E$ from $v$ to $w$ is equal to the cardinality of the set of edges in $F$ from $v$ to $w$.  Thus $E \cong F$ as polymorphisms.
\end{proof}

\begin{prp} \label{bimodule-functor-prop}
Let $E := (U, V, E^1, r_E, s_E)$ be a polymorphism from $U$ to $V$, and let $F := (V, W, F^1, r_F, s_F)$ be a polymorphism from $V$ to $W$. Then ${kE^1} \otimes_{kV} {kF^1} \cong k({E \times F})^1$ as $kU$--$kW$-bimodules.
\end{prp}

\begin{proof}
The vector space $kE^1$ has $E^1$ as a basis, and the vector space $kF^1$ has $F^1$ as a basis.  Thus any element of ${kE^1}$ may be uniquely written as $\sum_{e \in E^1} \lambda_e e$ with $\lambda_e \in k$ for all $e \in E^1$, and any element of ${kF^1}$ may be uniquely written as $\sum_{f \in F^1}\mu_f f$ with $\mu_f \in k$ for each $f \in F^1$.  Define $\phi : {kE^1} \times {kF^1} \to k(E \times F)^1$ by
$$ \phi \left( \sum_{e \in E^1} \lambda_e e, \sum_{f \in F^1} \mu_f f \right) := \sum_{ \{ (e,f) \in E^1 \times F^1 : r_E(e) = s_F(f) \} } \lambda_e \mu_f (e,f).$$
One can verify that this map is a $kV$-balanced $kU$-$kW$-bilinear map from ${kE^1} \times {kF^1}$ to $k (E \times F)^1$.  Thus by the universal property of the tensor product, $\phi$ induces a $kU$-$kW$-linear map $\widetilde{\phi} :  {kE^1} \otimes_{kV} {kF^1} \to k(E \times F)^1$ with
$$ \widetilde{\phi} \left( \sum_{e \in E^1} \lambda_e e \otimes \sum_{f \in F^1} \mu_f f \right) := \sum_{ \{ (e,f) \in E^1 \times F^1 : r_E(e) = s_F(f) \} } \lambda_e \mu_f (e,f).$$
Furthermore, since $\{ (e,f) \in E^1 \times F^1 : r_E(e) = s_F(f) \}$ is a basis for the $k$-vector space $k(E \times F)^1$, we may define a $k$-linear map $\psi : k(E \times F)^1 \to {kE^1} \otimes_{kV} {kF^1}$ with the property that $\psi (e,f) = e \otimes f$ for all $(e,f)$ in this basis.  Moreover, $\psi$ is also a $kU$-$kW$-module map.

One can verify that $\psi \circ \widetilde{\phi} = \operatorname{id}_{{kE^1} \otimes_{kV} {kF^1}}$.  (It suffices to check this equality on the generators $e \otimes f$ of ${kE^1} \otimes_{kV} {kF^1}$, realizing that when $r_E(e) \neq s_F(f)$ we have $e \otimes f = e r_E(e) \otimes f = e \otimes r_E(e) f = e \otimes 0 = 0$.)  One can also verify that  $\widetilde{\phi} \circ \psi = \operatorname{id}_{k(E \times F)^1}$.  (Again, it suffices to check this equality on the generators $ \{ (e,f) \in E^1 \times F^1 : r_E(e) = s_F(f) \}$.)  Hence $\widetilde{\phi}$ and $\psi$ are $kU$--$kW$-bimodule isomorphisms, and ${kE^1} \otimes_{kV} {kF^1} \cong k(E \times F)^1$.
\end{proof}

We now prove the first of our two main results.  We note that an analogous result in the context of C$^*$-algebras has been established in \cite[Proposition 3.5]{CDE}.  

\begin{thm}[Graph Shift Equivalence if and only if Module Shift Equivalence]
\label{thm:shift-equivalence-module}
Let $k$ be a field, and let $E$ and $F$ be finite graphs with no sinks.  Let ${kE^1}$ (respectively, ${kF^1}$) denote the $kE^0$--$kE^0$ (respectively, $kF^0$--$kF^0$) bimodule described in Definition~\ref{graph-bimodule-def}.   Then the adjacency matrices of $E$ and $F$ are shift equivalent if and only if the following condition holds:
\begin{itemize}
\item[${\rm (MSE)}$]  There exists a  $kE^0$--$kF^0$-bimodule $M$, a $kF^0$--$kE^0$-bimodule $N$, and a positive integer $n$ for which 
\begin{align*}
{(kE^1)}^{\otimes n} &\cong M \otimes_{kF^0} N & {(kF^1)}^{\otimes n} &\cong N \otimes_{kE^0} M &  \\
{kE^1} \otimes_{kE^0} M &\cong M \otimes_{kF^0} {kF^1} & {kF^1} \otimes_{kF^0} N &\cong N \otimes_{kE^0} {kE^1},
\end{align*}
where the isomorphisms are bimodule isomorphisms.
\end{itemize}
\end{thm}

\begin{proof}  
${\rm (SE)} \implies {\rm (MSE)}$.  Let $E$ and $F$ be graphs with adjacency matrices $A$ and $B$, respectively.  Then $A$ and $B$ are shift equivalent by hypothesis, so there exist matrices $R$ and $S$ and a positive integer $n$ such that
$$A^n = RS , \ \ B^n = SR, \ \ A R = R B, \ \  \mbox{and} \ \ \ B S = S A.$$
As described in Definition~\ref{adjacency-matrix-def}, there is a finite polymorphism $G$ from $E^0$ to $F^0$ with adjacency matrix $R$, and there is a finite polymorphism $H$ from $F^0$ to $E^0$ with adjacency matrix $S$.  
Also as described in Definition~\ref{adjacency-matrix-def}, these matrix equations are equivalent to the existence of the following isomorphisms of polymorphisms:
$$E^{\times n} \cong G \times H , \ \ F^{\times n} \cong H \times G, \ \ E \times G \cong G \times F, \ \  \mbox{and} \ \ F \times H \cong H \times E.$$
Forming the associated polymorphism bimodules yields bimodule isomorphisms
\begin{align*}
 k (E^{\times n})^1 & \cong k(G \times H)^1 &  k(F^{\times n})^1 &\cong k(H \times G)^1\\
 k(E \times G)^1 &\cong k(G \times F)^1 & k(F \times H)^1 &\cong k(H \times E)^1, 
\end{align*}
and Proposition~\ref{bimodule-functor-prop} gives bimodule isomorphisms
\begin{align*}
(kE^1)^{\otimes n} &\cong kG^1 \otimes kH^1 & (kF^1)^{\otimes n} &\cong kH^1 \otimes kG^1 \\
kE^1 \otimes kG^1 &\cong kG^1 \otimes kF^1 & kF^1 \otimes kH^1 &\cong kH^1 \otimes kE^1.
\end{align*}
Thus ${\rm (MSE)}$ follows by setting $M := kG^1$ and $N := kH^1$.

${\rm (MSE)} \implies {\rm (SE)}$.  By Proposition~\ref{bimodule-include-prop} there exists a finite polymorphism $G$ from $E^0$ to $F^0$ with $kG^1 \cong M$, and there exists a finite polymorphism $H$ from $F^0$ to $E^0$ with $kH^1 \cong N$.  Thus ${\rm (MSE)}$ gives   
\begin{align*}
(kE^1)^{\otimes n} &\cong kG^1 \otimes kH^1 & (kF^1)^{\otimes n} &\cong kH^1 \otimes kG^1 \\
kE^1 \otimes kG^1 &\cong kG^1 \otimes kF^1 & kF^1 \otimes kH^1 &\cong kH^1 \otimes kE^1, 
\end{align*}
and Proposition~\ref{bimodule-functor-prop} implies 
\begin{align*}
 k (E^{\times n})^1 & \cong k(G \times H)^1 &  k(F^{\times n})^1 &\cong k(H \times G)^1 \\
 k(E \times G)^1 &\cong k(G \times F)^1 & k(F \times H)^1 &\cong k(H \times E)^1.
\end{align*}
By the uniqueness assertion established in  Proposition~\ref{bimodule-include-prop}, we obtain isomorphisms of polymorphisms  
$$E^{\times n} \cong G \times H , \ \ F^{\times n} \cong H \times G, \ \ E \times G \cong G \times F, \ \  \mbox{and} \ \ F \times H \cong H \times E.$$
Taking the adjacency matrices of these polymorphisms and using the properties described in Definition~\ref{adjacency-matrix-def}, we obtain
$$A_E^n = A_G A_H , \ \ A_F^n = A_H A_G, \ \ A_E A_G = A_G A_F, \ \  \mbox{and} \ \ \ A_F A_H = A_H A_E.$$
Thus the adjacency matrices $A_E$ and $A_F$ are shift equivalent.
\end{proof}

\begin{rmk}\label{BiModremark}
As described in Remark~\ref{polymorphism-category-rem}, there is a category $\mathsf{Poly}$ whose objects are finite sets and whose morphisms are isomorphism classes of polymorphisms, using the product of polymorphisms for composition.  In addition, one may consider the category $\mathsf{BiMod}$ whose objects are (unital) rings, and for rings $R$ and $S$ the morphisms from $R$ to $S$ consist of isomorphism classes of $R$--$S$-bimodules with composition given by tensoring; i.e., if $M$ is an $R$--$S$-bimodule and $N$ is an $S$--$T$-bimodule, then $[N] \circ [M] := [ M \otimes_S N ]$.

Forming the bimodule of a polymorphism may be viewed as a functor from $\mathsf{Poly}$ to $\mathsf{BiMod}$.  More precisely, on objects the functor assigns $V \mapsto kV$ (where we view $kV$ as a ring), and on morphisms the functor assigns $[E] \mapsto [kE^1]$.  Proposition~\ref{bimodule-functor-prop} is exactly the statement that this assignment is functorial on morphisms.  Proposition~\ref{bimodule-include-prop} implies that this functor is faithful and that it maps onto the full subcategory of $\mathsf{BiMod}$ whose objects are rings that are (isomorphic to) finite direct sums of the field $k$, and whose morphisms are bimodules between these rings.
\end{rmk}

\begin{rmk} \label{data-first-appears-rem}
Theorem~\ref{thm:shift-equivalence-module} allows us to replace Condition~{\rm (SE)} in Hazrat's Conjecture with Condition~${\rm (MSE)}$.  This is useful for creating variants of Hazrat's Conjecture.  Condition~{\rm (SE)} asserts that certain matrix products are equal, while Condition~${\rm (MSE)}$ requires certain bimodules to be isomorphic.  In particular, Condition~${\rm (MSE)}$ allows us to add hypotheses on the bimodule isomorphisms, so we can explore stronger hypotheses that would require the bimodules occurring to be ``isomorphic in certain ways".  This is precisely the approach we shall follow to achieve Theorem~\ref{thm:main2}, in which we show that Condition~${\rm (MSE)}$ together with the additional hypothesis that the isomorphisms make certain diagrams commute imply that the Leavitt path algebras of the graphs are graded Morita equivalent.  To accomplish this, we will build an algebraic analog of the observation  that when $k = \mathbb{C}$, each of the isomorphisms ${kE^1} \otimes_{kE^0} M \cong M \otimes_{kF^0} {kF^1}$ and ${kF^1} \otimes_{kF^0} N \cong N \otimes_{kE^0} {kE^1}$ is precisely the data used by Ery\"uzl\"u in \cite{Eryuzlu} to create a $C^*$-correspondence in the $C^*$-algebra setting.   A similar conclusion, using different techniques, was also achieved in \cite{MeyerSehnem}.  
\end{rmk}

\section{The Bridging Bimodule} \label{sec-const-module}

In this section we use the constructions described previously to define the key tool in our investigation, the {\it bridging bimodule}.  

\begin{dfn} \label{conjugacy-def}
Let $E$ and $F$ be finite graphs with no sinks.  A $kE^0$--$kF^0$-bimodule $M$ is called an \emph{$E$--$F$-conjugacy} (or simply a \emph{conjugacy} if $E$ and $F$ are understood) in case  $$kE^1 \otimes_{kE^0} M \cong M \otimes_{kF^0} kF^1$$
as $kE^0$--$kF^0$-bimodules.
\end{dfn}

Observe that Condition~${\rm (MSE)}$ from Theorem~\ref{thm:shift-equivalence-module} presumes  that there is both a conjugacy from $E$ to $F$ as well as a conjugacy from $F$ to $E$.  

\begin{rmk}
Let $\mathsf{BiMod}$ denote the category whose objects are rings and whose morphisms are isomorphism classes of bimodules, with composition given by  tensor product.  In Definition~\ref{conjugacy-def} a conjugacy bimodule $M$ can be considered as a morphism in $\mathsf{BiMod}$ that makes the following diagram commute:
$$
\xymatrix{
kE^0 \ar[r]^{kE^1} \ar[d]_M & kE^0 \ar[d]^M \\
kF^0 \ar[r]^{kF^1} & kF^0
}
$$

\noindent In the subject of dynamics such a morphism is sometimes called a ``conjugacy".  Similarly, in a group we often say elements $a$ and $b$ are ``conjugate" if there exists an element $g$ such that $a = gbg^{-1}$.  This relation is generalized to semigroups by calling elements $a$ and $b$ ``conjugate" when there exists a (not necessarily invertible) element $g$ such that $ag = gb$.   This last equation is analogous to our relation $kE^1 \otimes_{kE^0} M \cong M \otimes_{kF^0} kF^1$.  These observations provide  motivation for our use of the term ``conjugacy" in Definition~\ref{conjugacy-def}.
\end{rmk}

For our construction it will be important to not only have a conjugacy $M$ but to also choose a particular isomorphism $\sigma: kE^1 \otimes_{kE^0} M \to M \otimes_{kF^0} kF^1$.

\begin{dfn}
Let $E$ and $F$ be finite graphs with no sinks.  If $M$ is an $E-F$ conjugacy, so that $kE^1 \otimes_{kE^0} M \cong M \otimes_{kF^0} kF^1$, an \emph{isomorphism implementing the conjugacy $M$} is a bimodule isomorphism
$$ \sigma : kE^1 \otimes_{kE^0} M \to M \otimes_{kF^0} kF^1.$$
We call the pair $(M, \sigma)$ a \emph{specified conjugacy pair} (or \emph{specified conjugacy} for short) \emph{from} $E$  \emph{to} $F$.
\end{dfn}

\begin{rmk} \label{left-action-rem}
Given a specified conjugacy pair $(M, \sigma)$ from $E$ to $F$, we shall construct an $L_k(E)$--$L_k(F)$-bimodule $Y_{(M,\sigma)}$.  We do this by defining $Y_{(M,\sigma)}$ to be $M \otimes_{kF^0} L_k(F)$.  The tensor product $M \otimes_{kF^0} L_k(F)$ naturally has a right $L_k(F)$-action, as well as a natural left $kE^0$-action.  But it shall take some effort for us to show that there is in fact a  left $L_k(E)$-action on $M \otimes_{kF^0} L_k(F)$ that makes this tensor product into an $L_k(E)$--$L_k(F)$-bimodule.  Showing that this left $L_k(E)$-action exists is the goal of this section, and after some preliminary results we accomplish this in Theorem~\ref{thm:bimodule-action}.
\end{rmk}

\begin{exm}\label{identityconjpair}
If $E$ is a finite graph with no sinks, we define $\operatorname{id}_E : E \to E$ by setting 
$$ \operatorname{id}_E := (kE^0, \epsilon_E)$$
where $\epsilon_E : kE^1 \otimes_{kE^0} kE^0 \to kE^0 \otimes_{kE^0} kE^1$ is the bimodule isomorphism with
$$ \epsilon_E (e \otimes v) := s(e) \otimes ev.$$
(Note that $e \otimes v$ is zero unless $v = r(e)$, so $\epsilon_E$ is taking $e \otimes r(e)$ to $s(e) \otimes e$.) 
Specifically,  $\operatorname{id}_E = (kE^0, \epsilon_E)$ is a specified conjugacy pair from $E$ to $E$.    
\end{exm}

\begin{exm}
Consider the graphs $E $ and $F$ given by 
$$\xymatrix{
 \bullet^{v} \ar@(ld, lu)[]^{e_1}  \ar@(rd, ru)[]_{e_2} 
}
\qquad \text{and} \qquad 
\xymatrix{
\bullet^{x} \ar@(ld, lu)[]^{f_1}   \ar@/^/[r]^{g_1} & \bullet^{y}  \ar@(rd, ru)[]_{f_2}  \ar@/^/[l]^{g_2}}$$
respectively. Note that $A_E = RS = (2)$ and $A_F = SR  = \left(\begin{smallmatrix} 1 & 1 \\ 1 & 1 \end{smallmatrix}\right)$, where $R = \begin{pmatrix} 1 & 1 \end{pmatrix}$ and $S = \left(\begin{smallmatrix}1 \\ 1 \end{smallmatrix}\right)$.  Let $G$ be the polymorphism $(\{v\}, \{x, y\} , \{e_{v,x}, e_{v, y} \}, r, s)$ defined by $R$ and let $H$ be the polymorphism $( \{x, y\}, \{v\} , \{ f_{x, v}, f_{y, v}\}, r, s)$ defined by $S$.  A computation shows that there are bimodule isomorphisms $\psi_E \colon kG^1 \otimes_{kF^0} kH^1 \to kE^1$ and $\psi_F : kH^1\otimes_{kE^0} k G^1 \to kF^1$ such that 
$$\psi_E ( e_{v,x} \otimes f_{x, v} ) =  e_1 \quad \text{and} \quad \psi_E ( e_{v,y} \otimes f_{y, v} ) =  e_2$$
and
\begin{align*}
\psi_F( f_{x, v} \otimes e_{v,x} ) &= f_1  & \psi_F( f_{y, v} \otimes e_{v,y} ) &= f_2 \\
\psi_F( f_{x, v} \otimes e_{v,y} ) &= g_1  & \psi_F( f_{y, v} \otimes e_{v,x} ) &= g_2.
\end{align*}
Let $\sigma_R \colon k E^1 \otimes_{kE^0} kG^1 \to kG^1 \otimes_{kF^0} kF^1$ denote the composition of isomorphisms
\begin{align*}
kE^1 \otimes_{kE^0} kG^1 \cong_{ \psi_E^{-1} \otimes \mathrm{id} } & ( k G^1  \otimes_{kF^0 } kH^1 ) \otimes_{ kE^0 } k G^1 \\
&\cong_{\alpha_{kG^1, kH^1, kG^1} }  k G^1 \otimes_{kF^0 } ( kH^1  \otimes_{ kE^0 } k G^1 ) \cong_{\mathrm{id} \otimes \psi_F } kG^1 \otimes_{kF^0 } k F^1
\end{align*}
and let $\sigma_S \colon k F^1 \otimes_{ kF^0 } kH^1 \to kH^1 \otimes_{kE^0 } kE^1$ denote the composition of isomorphisms
\begin{align*}
kF^1 \otimes_{kF^0} kH^1 \cong_{ \psi_F^{-1} \otimes \mathrm{id} } & ( k H^1 \otimes_{kG^0 } kG^1 ) \otimes_{ kF^0 } k H^1 \\
&\cong_{\alpha_{kH^1, kG^1, kH^1 }}  k H^1 \otimes_{kE^0 } ( kG^1  \otimes_{ kF^0 } k H^1 ) \cong_{\mathrm{id} \otimes \psi_E } kH^1 \otimes_{kE^0 } k E^1.
\end{align*}
Thus  $( kG^1, \sigma_R)$ is a specified conjugacy pair from $E$ to $F$,  and $(kH^1, \sigma_S)$ is a specified conjugacy pair from $F$ to $E$.

\end{exm}

The hypothesis  that  a graph has no sinks is used often throughout the article, the primary reason being that it affords  applications of the following result.  

\begin{lem}\label{nosinkslemma}
Let $E$ be a finite graph with no sinks.   
\begin{enumerate}
\item   The identity element $1_{L_k(E)}$ of $L_k(E)$ can be written as 
$$1_{L_k(E)} = \sum_{e\in E^1} ee^*.$$

\item For each $v\in E^0$ and each  positive integer $m$, 
$$ 
v =  \sum_{ \lambda_1 \cdots \lambda_m \in  v E^m } \lambda_1 \cdots \lambda_m (\lambda_1\cdots \lambda_m)^* .
$$

\end{enumerate}

\end{lem}
\begin{proof}
Since $  E$ is finite with no sinks we have $v = \sum_{s(e) = v} ee^*$ for all $v \in E^0$.  For (1), we have $1_{L_k(E)} = \sum_{v \in E^0}v =  \sum_{v \in E^0} (\sum_{s(e) = v} ee^* )= \sum_{e \in E^1} ee^*$.   For (2), since $E$ has no sinks, the (CK2)  relation gives $v = \sum_{e \in vE^1} ee^*$ for all $v \in E^0$.  For $m \geq 1$, we use the fact that $ee^* = e r(e) e^*$ for each $e\in E^1$, and then replace $r(e)$ with the appropriate (CK2)  relation as many times as required.  
\end{proof}

\begin{prp}\label{lem:iso-YotimesOY}
Let $E$ be a finite graph with no sinks.  Let $m$ be any positive integer.  Then there exists a $kE^0$--$L_k(E)$-bimodule isomorphism
\[
\rho_{E,m} : \  (kE^1)^{\otimes m} \otimes_{ kE^0 } L_k(E)  \ \to \ L_k(E)
\]
such that $\rho_{E,m}((x_1 \otimes \cdots \otimes x_m) \otimes S)=x_1x_2\cdots x_m S$ for all $x_1, \ldots, x_m \in kE^1$ and for all $S \in L_k(E)$.  
\end{prp}

\begin{proof}
Define $\rho_{E,m}$ by setting 
$$\rho_{E,m} \left( \sum_{ i = 1}^N ( x_{1, i} \otimes \cdots \otimes x_{m,i} ) \otimes S_i  \right) = \sum_{ i = 1}^N x_{1,i} \cdots x_{m,i} S_i.$$
  It is straightforward to check that $\rho_{E,m}$ is a well-defined $k$-linear map.  In addition, for any $a \in kE^0$ we have 
\begin{align*}
    \rho_{E,m}( a ( ( x_1 \otimes \cdots \otimes x_m) \otimes S) ) &= \rho_{E,m}( ((ax_1)\otimes x_2 \otimes \cdots \otimes x_m) \otimes S ) = (ax_1) x_2 \cdots x_m S \\
    &= a( x_1 \cdots x_m S) = a\rho_{E,m}( (x_1 \otimes \cdots \otimes x_m) \otimes S),   
    \end{align*}
    and for any $T \in L_k(E)$ we have
    \noindent
    \begin{align*}
    \rho_{E,m}( ( x_1 \otimes \cdots \otimes x_m) \otimes S ) T) &= \rho_{E,m} ( (x_1 \otimes \cdots \otimes x_m) \otimes (ST) ) = x_1 \cdots x_m (ST) \\
                &= (x_1 \cdots x_m S)T =\rho_{E,m}( (x_1 \otimes \cdots \otimes x_m) \otimes S)T.
\end{align*}
Thus $\rho_{E,m}$ is a $kE^0$--$L_k(E)$-bimodule map.

Let $\mu$ and $\nu$ be  paths in $E$.  Since $E$ is a finite graph with no sinks, Lemma~\ref{nosinkslemma}(2)  implies
\[
s(\mu)=  \sum_{ \lambda_1 \cdots \lambda_m \in s(\mu) E^m } \lambda_1 \cdots \lambda_m (\lambda_1\cdots \lambda_m)^*\
\]
in $L_k(E)$.  We then have
\begin{align*}
\rho_{E,m} \Big(  \sum_{ \lambda_1 \cdots  \lambda_m \in s(\mu) E^m } ( \lambda_1 &\otimes \cdots \otimes \lambda_m ) \otimes (\lambda_1\cdots \lambda_m)^* \mu \nu^* \Big)\\
 &=  \sum_{ \lambda_1 \cdots \lambda_m \in s(\mu) E^m }  \lambda_1 \cdots \lambda_m (\lambda_1\cdots \lambda_m)^*\mu \nu^* \\
&= s( \mu) \mu \nu^* \\
&= \mu \nu^*.
\end{align*}
Since $L_k(E)$ is the linear span of elements of the form $\mu \nu^*$ for paths $\mu, \nu$ in $E$, we conclude $\rho_{E,m}$ is surjective.

We note the following for use in the displayed equations below:  if $x,y \in kE^1$, then the element $x^*y$ of $L_k(E)$ is in $kE^0$ (because $x$ and $y$ are $k$-linear combinations of edges, and $e_1^*e_2$ is either $0$ or a vertex for all $e_1, e_2 \in E^1$).  In particular,  
\begin{equation} \label{useful-below-eq}
e (e^*u) \otimes_{kE^0} W=  e \otimes_{kE^0} (e^*u) W
\end{equation} 
in  $kE^1 \otimes_{ kE^0 }  L_k(E)$ for all $e\in E^1$, $u \in kE^1$, and $W \in L_k(E)$.

To prove $\rho_{E,m}$ is injective,  we note that it is straightforward to check the usual ``associativity of tensors" bimodule isomorphism 
$$T_m:  \ \ (kE^1)^{\otimes m} \otimes_{ kE^0 } L_k(E) \ \to \ (kE^1)^{\otimes (m-1)} \otimes_{ kE^0 } (kE^1 \otimes_{ kE^0 }  L_k(E))$$
has the property that for each $m\geq 2$, 
$$\rho_{E,m-1} \circ (\mathrm{id}_{(kE^1)^{\otimes (m-1)}} \otimes \rho_{E,1} ) \circ T_m = \rho_{E,m}.$$  Thus once we prove that $\rho_{E,1}$ is injective, it follows from an inductive argument that $\rho_{E,m}$ is injective for all $m \geq 1$.

To this end, suppose $z = \sum_{ i = 1}^N u_i \otimes W_i \in kE^1 \otimes_{ kE^0 }  L_k(E)$ with $\rho_{E,1}(z)=0$.  Then $\sum_{i=1}^N u_i W_i = 0$, which implies that 
\begin{equation} \label{useful-below-soon-eq}
\sum_{ i=1}^N (x^*u_i)W_i =0 \qquad \text{for all $x \in kE^1$.}
\end{equation}
Note that $(\sum_{v \in E^0} v) u_i  = 1_{L_k(E)} u_i= u_i$ for all $i$.  By Lemma~\ref{nosinkslemma} we obtain 
\[
u_i = \left(\sum_{v \in E^0} v\right) u_i  = \sum_{ v \in E^0 } \sum_{ e \in s^{-1}(v)} e e^* u_i = \sum_{ e \in E^1} ee^* u_i
\]
for all $1 \leq i \leq N$.  Therefore, using \eqref{useful-below-eq} we conclude $e (e^*u_i) \otimes W_i  - e \otimes (e^*u_i) W_i = 0$ for all $1 \leq i \leq N$, and hence
\begin{align*}
0&=\sum_{ i = 1}^N \sum_{ e \in E^1} \left( e (e^*u_i) \otimes W_i  - e \otimes (e^*u_i) W_i \right) \\
    &= \sum_{ i = 1}^N \sum_{ e \in E^1 } ee^*u_i \otimes W_i - \sum_{ i = 1}^N \sum_{ e \in E^1} e \otimes (e^*u_i) W_i  \\
&= \sum_{ i =1}^N  u_i \otimes W_i - \sum_{ e \in E^1} \sum_{ i =1}^N e \otimes (e^*u_i)W_i \\
&= z - \sum_{ e \in E^1} e \otimes\left(\sum_{ i =1}^N  (e^*u_i)W_i \right) \\
&= z - \sum_{ e \in E^1} e \otimes 0 \qquad \qquad \text{(by \eqref{useful-below-soon-eq})} \\
&=z.
\end{align*}
Consequently, $\rho_{E,1}$ is injective, which suffices to yield that  $\rho_{E,m}$ injective for all $m\geq  1$.
\end{proof}

\begin{ntn} \label{rho-E-not}
We shall denote the isomorphism $\rho_{E,1}$ in Proposition~\ref{lem:iso-YotimesOY} by $\rho_E$.  In particular, $$\rho_E : kE^1 \otimes_{kE^0} L_k(E) \to L_k(E) \ \ \  \ \mbox{satisfies} \ \ \ \  \rho_E (x \otimes S) = xS$$
for all $x\in kE^1$ and all $S\in L_k(E)$.  
\end{ntn}

As a result of Proposition~\ref{lem:iso-YotimesOY}, for each positive integer $m$ there is   a  $kE^0$--$L_k(E)$-bimodule isomorphism $\rho_{E,m} :    (kE^1)^{\otimes m} \otimes_{ kE^0 } L_k(E)   \to  L_k(E)$.  In order to construct the left $L_K(E)$-action on $M \otimes_{kF^0} L_k(F)$ we shall only need to use $\rho_E := \rho_{E,1}$.  However, in a later section we shall need the maps $\rho_{E,m}$ for each positive integer $m$.  This is why we established Proposition~\ref{lem:iso-YotimesOY} for general $m$.

\begin{ntn}\label{thetaandmunotation}
Let $E$ be a finite graph and let $M$ be a $kE^0$--$kF^0$-bimodule.  For each $e \in E^1$ let $\theta_e \colon M \to kE^1 \otimes_{ kE^0 } M$ be the $kE^0$--$kF^0$-bimodule homomorphism with
$$\theta_e( m ) = e \otimes m$$
for all $m \in M$, and let $\theta_{e^*} \colon  kE^1 \otimes_{ kE^0 } M \to M$ be the $kE^0$--$kF^0$-bimodule homomorphism with
$$
\theta_{e^*}( x \otimes m) = (e^*x) m
$$
for all $m \in M$ and $x\in  kE^1$.

If $A$ and $B$ are rings and $N$ is an $A$--$B$-bimodule, then for each $a \in A$ we let $
\mu_a \colon N \to N$ denote the right $B$-endomorphism defined by left multiplication by $a$; that is, 
$$\mu_a(n) = an$$ 
for all $a \in A$ and $n \in N$.
\end{ntn}

We are now in  position to construct the left $L_k(E)$-action on $M \otimes_{kF^0} L_k(F)$ mentioned in Remark~\ref{left-action-rem}.

\begin{thm} \label{thm:bimodule-action}
Let $E$ and $F$ be finite graphs with no sinks.  Let $(M, \sigma)$ be a specified conjugacy pair from $E$ to $F$; that is, $M$ is a $kE^0$--$kF^0$-bimodule, and
$$
\sigma \colon \  kE^1 \otimes_{kE^0} M \ \to \ M \otimes_{kF^0} kF^1
$$
is a $kE^0$--$kF^0$-bimodule isomorphism.  The right $L_k(F)$-module $$M \otimes_{kF^0} L_k(F)$$  admits a left $L_k(E)$-action which makes $M \otimes_{kF^0} L_k(F)$ into an $L_k(E)$--$L_k(F)$-bimodule.  The 
left $L_k(E)$-action is determined by the following relations:
\begin{itemize}
\item[(1)] If $v \in E^0$, $m \in M$, and $S \in L_k(F)$, then $$v \cdot (m \otimes S) = vm \otimes S.$$
\item[(2)] If $e \in E^1$, $m \in M$, and $S \in L_k(F)$ with $\sigma (e \otimes m) = \sum_{i=1}^n m_i \otimes y_i$ for $m_1, \ldots, m_n \in M$ and $y_1, \ldots, y_n \in kF^1$, then
$$ e \cdot (m \otimes S) = \sum_{i=1}^n m_i \otimes y_iS.$$
\item[(3)] If $e \in E^1$, $m \in M$, and $S \in L_k(F)$, and if for each $f \in F^1$ we have $\sigma^{-1} (m \otimes f) = \sum_{i=1}^{n_f} y_i^f \otimes m_i^f$ for $y_1^f, \ldots, y_{n_f}^f \in kE^1$ and $m_1^f, \ldots, m_{n_f}^f \in M$, then
$$ e^* \cdot (m \otimes S) =\sum_{f \in F^1} \sum_{i=1}^{n_f} (e^* y_i^f) m_i^f \otimes f^* S.$$
(Note, in particular, that if $S=g \in F^1$ is an edge and $\sigma^{-1}(m \otimes g) = \sum_{i=1}^{n_g} y_i^g \otimes m_i^g$, then $e^* \cdot (m \otimes g) = \sum_{i=1}^{n_g} e^* y_i^g m_i^g \otimes r(g)$.)
\end{itemize}
Moreover, if we define
$$Y_n := \operatorname{span}_k\{  x \otimes S \ | \ x\in M \text{ and } S \in L_k(F)_n\},$$ 
then the $k$-subspaces $\{Y_n \ | \ n\in \mathbb{Z} \}$ provide a $\ZZ$-grading that makes $M \otimes_{kF^0} L_k(F)$ a $\ZZ$-graded $L_k(E)$--$L_k(F)$-bimodule.
\end{thm}

\begin{proof}
Since the graphs have no sinks,  the map $\rho_F :  kF^1 \otimes_{kF^0} L_k(F) \to L_k(F)$ described in Notation~\ref{rho-E-not} and Proposition~\ref{lem:iso-YotimesOY} is a $kF^0$--$L_k(F)$-bimodule isomorphism.  For $v \in E^0$ and $e \in E^1$ we define the following elements of $\mathrm{End}_{L_k(F) } ( M \otimes_{kF^0} L_k(F) )$, using the associativity isomorphisms as described in Notation~\ref{associativity-iso-notation} and the maps $\theta$ and $\mu$ of Notation~\ref{thetaandmunotation}.
   \begin{eqnarray*}
      P_v & := & \mu_v \otimes \mathrm{id}_{L_k(F)} \\
   S_e & := &  (\mathrm{id}_M \otimes \rho_F) \circ  \alpha_{M,kF^1, L_k(F)}  \circ ( \sigma  \otimes \mathrm{id}_{L_k(F)}) \circ (\theta_e \otimes \mathrm{id}_{L_k(F)} ) \\
   & =  &(\mathrm{id}_M \otimes \rho_{F}) \circ  \alpha_{M,kF^1, L_k(F)} \circ ( \sigma \theta_e \otimes \mathrm{id}_{L_k(F)} ) \hspace{1in} \ \ \mbox{ and }    \\
    S_{e^*} & := &(\theta_{e^*} \otimes \mathrm{id}_{ L_k(F) } ) \circ (\sigma^{-1}  \otimes \mathrm{id}_{L_k(F)}) \circ  \alpha_{M,kF^1, L_k(F)}^{-1} \circ ( \mathrm{id}_M \otimes \rho_F^{-1}) \\
  &=&(\theta_{e^*} \sigma^{-1} \otimes \mathrm{id}_{ L_k(F) } ) \circ \alpha_{M,kF^1, L_k(F)}^{-1} \circ(\mathrm{id}_M \otimes \rho_F^{-1}). 
   \end{eqnarray*}
We shall show that the subset $\{ P_v, S_e, S_{e^*} : v \in E^0, e \in E^1 \}$ of $\mathrm{End}_{L_k(F) } ( M \otimes_{kF^0} L_k(F) )$ is a Cuntz-Krieger $E$-family.    Let $v, w \in E^0$ and $e \in E^1$.    There are five relations to verify.   (To ease notational burden we suppress the $\circ$ symbol when utilizing the composition of functions in these verifications.)  
  
  \medskip

\noindent \textsc{Relation~1:}  For $m \in M$ and $S \in L_k(F)$ we have
\begin{align*}
    P_vP_w (m \otimes S) &= (\mu_v \otimes \mathrm{id}_{L_k(F)} )  (\mu_w \otimes \mathrm{id}_{L_k(F)} ) (m \otimes S) = (\mu_v \otimes  \mathrm{id}_{L_k(F)} ) ( wm \otimes S) \\
    &= (vw m  \otimes S) =  \delta_{v,w} ( v m  \otimes S) = \delta_{v,w}  P_v ( m \otimes S), 
\end{align*}    
so that $P_v P_w = \delta_{v,w} P_v$.

\medskip

\noindent \textsc{Relation~2:}  For $m \in M$ and $S \in L_k(F)$ we have
\begin{align*}
    P_{s(e)} S_e ( m \otimes S ) 
    &= (\mu_{s(e)} \otimes \mathrm{id}_{L_k(F) }) (\mathrm{id}_M \otimes \rho_F) \alpha_{M,kF^1, L_k(F)}( \sigma \theta_e \otimes \mathrm{id}_{L_k(F)} ) ( m \otimes S) \\
    &= (\mu_{s(e)} \otimes \mathrm{id}_{L_k(F) }) (\mathrm{id}_M \otimes \rho_F ) \alpha_{M,kF^1, L_k(F)}( \sigma(  e \otimes m) \otimes S ) \\
    &= (\mathrm{id}_M \otimes \rho_F ) ( \mu_{s(e)} \otimes \mathrm{id}_{ kF^1 } \otimes \mathrm{id}_{ L_k(F)} ) \alpha_{M,kF^1, L_k(F)}( \sigma(  e \otimes m) \otimes S )\\
        &= (\mathrm{id}_M \otimes \rho_F ) ( s(e) \alpha_{M,kF^1, L_k(F)}( \sigma(  e \otimes m) \otimes S )) \\
    &= (\mathrm{id}_M \otimes \rho_F ) (  \alpha_{M,kF^1, L_k(F)} (\sigma( s(e)e \otimes m ) \otimes S)) \\
    &= (\mathrm{id}_M \otimes \rho_F )( \alpha_{M,kF^1, L_k(F)} ( \sigma( e \otimes m ) \otimes S) ) \\
    &= (\mathrm{id}_M \otimes \rho_F )( \alpha_{M,kF^1, L_k(F)} ( \sigma( \theta_e(m) ) \otimes S) ) \\
    &= S_{e}( m \otimes S ),
\end{align*}
and
\begin{align*}    
    S_eP_{r(e)}(m \otimes S) &= (\mathrm{id}_M \otimes \rho_F) \alpha_{M,kF^1, L_k(F)} ( \sigma \theta_e \otimes \mathrm{id}_{L_k(F)} )( r(e) m \otimes S ) \\
    &= (\mathrm{id}_M \otimes \rho_F) \alpha_{M,kF^1, L_k(F)}( \sigma( e \otimes r(e) m ) \otimes S ) \\
    &= (\mathrm{id}_M \otimes \rho_F) \alpha_{M,kF^1, L_k(F)} ( \sigma( e r(e) \otimes m) \otimes S ) \\
    &= (\mathrm{id}_M \otimes \rho_{F}) \alpha_{M,kF^1, L_k(F)}  ( \sigma( e \otimes m) \otimes S ) \\
    &= (\mathrm{id}_M \otimes \rho_{F}) \alpha_{M,kF^1, L_k(F)}  ( \sigma( \theta_e (m)) \otimes S ) \\
    &= S_{e}( m \otimes S ).
\end{align*}
Consequently, $S_e P_{r(e)} = P_{s(e)} S_e = S_e$.   

\medskip

\noindent \textsc{Relation~3:}  Next we show that $P_{r(e)} S_{e^*} = S_{e^*} P_{s(e)} = S_{e^*}$.  Note that if $x \in kE^1$, then 
$$ \mu_{r(e)} \theta_{e^*} ( x \otimes m ) = \mu_{r(e)} ( (e^*x) m )  = ( r(e) (e^*x) ) m 
					= (e^*x) m = \theta_{e^*}(x \otimes m),  \ \ \mbox{and}$$
$$\theta_{e^*} \mu_{s(e)} ( x \otimes m) = \theta_{e^*}( s(e) x \otimes m ) 
							= (e^*  s(e) x )  m \\
							= (e^*x ) m \\
							= \theta_{e^*} ( x \otimes m). $$

\noindent Therefore $\mu_{r(e)} \theta_{e^*} = \theta_{e^*}$ and $\theta_{e^*} \mu_{s(e)} = \theta_{e^*}$. Hence
\begin{align*}
P_{r(e)} S_{e^*} &(m \otimes S)  \\
&=  (\mu_{r(e)} \otimes \mathrm{id}_{L_k(F) })S_{e^*} (m \otimes S)    \\
&= (\mu_{r(e)} \otimes \mathrm{id}_{L_k(F) })(\theta_{e^*} \otimes \mathrm{id}_{ L_k(F) } )(\sigma^{-1}  \otimes \mathrm{id}_{L_k(F)}) \alpha_{M,kF^1, L_k(F)}^{-1}( \mathrm{id}_M \otimes \rho_F^{-1})  (m \otimes S)\\
&= (\mu_{r(e)} \theta_{e^*} \otimes \mathrm{id}_{ L_k(F) } )(\sigma^{-1}  \otimes \mathrm{id}_{L_k(F)}) \alpha_{M,kF^1, L_k(F)}^{-1}( \mathrm{id}_M \otimes \rho_F^{-1})  (m \otimes S)\\
&= (\theta_{e^*} \otimes \mathrm{id}_{ L_k(F) } )(\sigma^{-1}  \otimes \mathrm{id}_{L_k(F)}) \alpha_{M,kF^1, L_k(F)}^{-1}( \mathrm{id}_M \otimes \rho_F^{-1}) (m \otimes S) \\
&= S_{e^*}(m\otimes S)
\end{align*}
and
\begin{align*}  
S_{e^*}  P_{s(e)} ( m \otimes S ) & =  S_{e^*} ( s(e) m \otimes S ) \\
&= (\theta_{e^*} \otimes \mathrm{id}_{ L_k(F) } )(\sigma^{-1}  \otimes \mathrm{id}_{L_k(F)}) \alpha_{M,kF^1, L_k(F)}^{-1}( \mathrm{id}_M \otimes \rho_F^{-1})  ( s(e) m \otimes S ) \\
&= (\theta_{e^*}\mu_{s(e)} \otimes \mathrm{id}_{ L_k(F) } ) (\sigma^{-1}  \otimes \mathrm{id}_{L_k(F)}) \alpha_{M,kF^1, L_k(F)}^{-1}( \mathrm{id}_M \otimes \rho_F^{-1})  ( m \otimes S )  \\
&= (\theta_{e^*} \otimes \mathrm{id}_{ L_k(F) } )(\sigma^{-1}  \otimes \mathrm{id}_{L_k(F)}) \alpha_{M,kF^1, L_k(F)}^{-1}( \mathrm{id}_M \otimes \rho_F^{-1}) (m\otimes S) \\
&= S_{e^*} ( m \otimes S ).\end{align*}
Thus $P_{r(e)} S_{e^*} = S_{e^*}$ and $S_{e^*}P_{s(e)} = S_{e^*}$.

\medskip

\noindent \textsc{Relation~4:}   We note first that for $e,e' \in E^1$, 
$$\theta_{e^*}\theta_{e'} \otimes \mathrm{id}_{ L_k(F) } (m \otimes S) = \theta_{e^*}(e' \otimes m) \otimes S = (e^*e')m \otimes S 
    = \delta_{e,e'} r(e) m \otimes S 
    =  \delta_{e,e'} P_{r(e)}(m \otimes S),$$
    so that $\theta_{e^*}\theta_{e'} \otimes \mathrm{id}_{ L_k(F) } =  \delta_{e,e'} P_{r(e)}$.  Hence
\begin{align*}
    S_{e^*}S_{e'}  &= (\theta_{e^*} \sigma^{-1} \otimes \mathrm{id}_{ L_k(F) } ) \alpha_{M,kF^1, L_k(F)}^{-1}(\mathrm{id}_M \otimes \rho_F^{-1}) (\mathrm{id}_M \otimes \rho_{F}) \alpha_{M,kF^1, L_k(F)}( \sigma \theta_{e'} \otimes \mathrm{id}_{L_k(F)} ) \\
                &= (\theta_{e^*} \sigma^{-1} \otimes \mathrm{id}_{ L_k(F) } ) ( \sigma \theta_{e'} \otimes \mathrm{id}_{L_k(F)} ) \\
                &= \theta_{e^*}\theta_{e'} \otimes \mathrm{id}_{ L_k(F) } =  \delta_{e,e'} P_{r(e)}.
\end{align*}

\noindent \textsc{Relation~5:}   For the final relation, we first observe
\begin{align*}
 & \sum_{ e \in s^{-1}(v) } S_e S_{e^*} \\
& = \sum_{ e \in s^{-1}(v) } (\mathrm{id}_M \otimes \rho_{F}) \alpha_{M,kF^1, L_k(F)}( \sigma \theta_e \otimes \mathrm{id}_{L_k(F)} )(\theta_{e^*} \sigma^{-1} \otimes \mathrm{id}_{ L_k(F) } ) \alpha_{M,kF^1, L_k(F)}^{-1}(\mathrm{id}_M \otimes \rho_F^{-1}) \\
&= \sum_{ e \in s^{-1}(v) } (\mathrm{id}_M \otimes \rho_{F}) \alpha_{M,kF^1, L_k(F)}( \sigma \theta_e \theta_{e^*} \sigma^{-1} \otimes \mathrm{id}_{L_k(F)}) \alpha_{M,kF^1, L_k(F)}^{-1}(\mathrm{id}_M \otimes \rho_F^{-1}) \\
&=(\mathrm{id}_M \otimes \rho_{F}) \alpha_{M,kF^1, L_k(F)} \left(  \sum_{ e \in s^{-1}(v) }    \sigma \theta_e \theta_{e^*} \sigma^{-1} \otimes \mathrm{id}_{L_k(F)} \right) \alpha_{M,kF^1, L_k(F)}^{-1}(\mathrm{id}_M \otimes \rho_F^{-1}).
\end{align*}

If $\sigma^{-1}( m \otimes f) = \sum_{ i = 1}^n e_i \otimes m_i$ for $e_1, \ldots, e_n \in E^1$ and $m_1, \ldots, m_n \in M$, then

\begin{align*}
&\left( \sum_{ e \in s^{-1}(v) } 
\sigma \theta_e \theta_{e^*} \sigma^{-1} \otimes \mathrm{id}_{L_k(F)} \right)    (m \otimes f) \otimes S) \\
&\qquad =  \sum_{ e \in s^{-1}(v) } \sum_{ i = 1}^n \sigma ( e \otimes  (e^*e_i)m_i ) \otimes S = \sum_{ e \in s^{-1}(v) } \sum_{ i = 1}^n \sigma ( e (e^*e_i) \otimes m_i ) \otimes S \\
&\qquad = \left( \sum_{ i = 1}^n \sigma \left( \sum_{ e \in s^{-1}(v) }  e (e^*e_i) \otimes m_i  \right) \right) \otimes S 
= \left( \sum_{ i = 1}^n \sigma ( v e_i \otimes m_i)  \right) \otimes S \\
&\qquad = v  \left( \sum_{ i = 1}^n \sigma( e_i \otimes m_i) \otimes S \right) =  v \left( \sigma \left( \sum_{ i = 1}^n  e_i \otimes m_i \right) \otimes S \right) \\
&\qquad = v ( (m \otimes f) \otimes S ) =  ( v(m \otimes f) \otimes S ) \\
&\qquad = (\mu_v \otimes \mathrm{id}_{kF^1}\otimes \mathrm{id}_{L_k(F)})(( m \otimes f) \otimes S).
\end{align*}
Hence,
\begin{align*}
 & \sum_{ e \in s^{-1}(v) } S_e S_{e^*} \\
&= (\mathrm{id}_M \otimes \rho_{F}) \alpha_{M,kF^1, L_k(F)}(\mu_v \otimes \mathrm{id}_{kF^1}\otimes \mathrm{id}_{L_k(F)}) \alpha_{M,kF^1, L_k(F)}^{-1}(\mathrm{id}_M \otimes \rho_F^{-1}) \\
&= P_v,
\end{align*}
as the final composition of functions is easily seen to simplify to $\mu_v \otimes \mathrm{id}_{L_k(F)}$.

\smallskip

\noindent Relations 1 to 5 verify that  the set $\{ P_v, S_e, S_{e^* } : v \in E^0, e \in E^1 \}$ is a Cuntz-Krieger $E$-family in $\mathrm{End}_{L_k(F) } ( M \otimes_{kF^1} L_k(F) )$.  

Consequently, by the universal property of $L_k(E)$ there exists a $k$-algebra homomorphism 
$$
\phi \colon L_k(E) \to \mathrm{End}_{L_k(F) } ( M \otimes_{kF^1} L_k(F) )
$$
for which $\phi(v) = P_v$, $\phi(e) = S_e$, and $\phi(e^*) = S_{e^*}$.  

Since $M$ is a left $kE^0$-module we have by definition that $1_{kE^0} m = m$ for all $m\in M$, so that   $ ( \sum_{v\in E^0}v) m = m$; this gives  $(\sum_{v\in E^0}\mu_v) (m) = m$, which implies that $$\sum_{v\in E^0}P_v = \sum_{v\in E^0} \mu_v \otimes  {\rm id}_{L_k(F)} =  {\rm id}_{M\otimes_{kF^1} L_k(F)}.$$   
Thus    $\phi(1_{L_k(E)}) (z)    =  \phi(\sum_{v\in E^0}v) (z) = (\sum_{v\in E^0} P_v) (z) = z$ for all $z\in  M \otimes_{kF^1} L_k(F)$.    
With this observation so noted, it is standard to conclude that the right $L_k(F)$-module $M \otimes_{kF^1} L_k(F)$ becomes an $L_k(E)$--$L_k(F)$-bimodule by defining the left $L_k(E)$-action as $$a \cdot z := \phi(a)(z)$$ for all  $a \in L_k(E)$ and $z \in M \otimes_{kF^1} L_k(F)$.

Since $\phi(v) = P_v$, $\phi(e) = S_e$, and $\phi(e^*) = S_{e^*}$, it is straightforward to verify (1)--(3) in the theorem's statement.

Finally, we verify the $k$-subspaces $\{ Y_n : n \in \mathbb{Z}\}$ provide a bimodule  $\mathbb{Z}$-grading.  It is straightforward to verify $M \otimes_{kF^1} L_k(F) = \bigoplus_{n\in \mathbb{Z}} Y_n$ as abelian groups.  Further, it is clear that $Y_n L_k(F)_m \subseteq Y_{n+m}$ for all $m,n \in \ZZ$.   In addition, the explicit description of the left module action of $L_k(E)$ on $M \otimes_{kF^1} L_k(F)$ given in (1)--(3) yields that $v Y_n \subseteq Y_n$, $eY_n \subseteq Y_{n+1}$, and $e^* Y_n \subseteq Y_{n-1}$ for all $n \in \mathbb{Z}$ and for all $v \in E^0$ and $e \in E^1$.  Since the elements $v, e, e^*$  generate $L_k(E)$ as a $k$-algebra, we conclude $L_k(E)_m Y_n  \subseteq Y_{n+m}$ for all $m,n \in \ZZ$.    
\end{proof}

\begin{dfn} \label{bridging-bimodule-def}
Let $E$ and $F$ be finite graphs with no sinks.  For a specified conjugacy pair $(M, \sigma)$, we shall use the notation
$$Y_{(M,\sigma)} := M \otimes_{kF^0} L_k(F)$$ 
to denote the $\mathbb{Z}$-graded $L_k(E)$--$L_k(F)$-bimodule of Theorem~\ref{thm:bimodule-action}.  We call $Y_{(M,\sigma)}$ the \emph{bridging bimodule for} $(M, \sigma)$.
\end{dfn}

\section{Properties of the Bridging Bimodule} \label{bimodule-key-properties-sec}

We start this section by describing  a notion of equivalence for specified conjugacy pairs,  and subsequently show that equivalent specified conjugacy pairs produce graded-isomorphic bridging bimodules.

\begin{dfn} \label{equivalent-conj-pairs-def}
Let $E$ and $F$ be finite graphs with no sinks.   Suppose $(M_1, \sigma_1)$ and $(M_2, \sigma_2)$ are specified conjugacy pairs from $E$ to $F$.  We say $(M_1, \sigma_1)$ is \emph{equivalent to}  $(M_2, \sigma_2)$ if there is a $kE^0$--$kF^0$-bimodule isomorphism $\phi \colon \ M_1 \to {M_2}$ for which the following diagram commutes:
\begin{equation}\label{eq-equivalent-conj-pairs-def}
\vcenter{
\xymatrix{
kE^1 \otimes_{ kE^0 } M_1 \ar[r]^-{\sigma_1} \ar[d]_{\mathrm{id}_{kE^1} \otimes \phi }& M_1 \otimes_{ kF^0 } kF^1 \ar[d]^{\phi \otimes \mathrm{id}_{kF^1} } \\
kE^1 \otimes_{ kE^0 } {M_2}  \ar[r]_-{\sigma_2} & {M_2} \otimes_{ kF^0 } kF^1.
}
}
\end{equation}
It is straightforward to verify that this is an equivalence relation on the set of specified conjugacy pairs from $E$ to $F$.
\end{dfn}

\begin{prp}\label{prop:MUF-iso}
Let $E$ and $F$ be finite graphs with no sinks, and let $(M_1, \sigma_1)$ and $(M_2, \sigma_2)$ be specified conjugacy pairs from $E$ to $F$.  If $(M_1, \sigma_1)$ is equivalent to  $(M_2, \sigma_2)$ in the sense of Definition~\ref{equivalent-conj-pairs-def}, then the bridging bimodules $Y_{(M_1, \sigma_1)}$ and $Y_{({M_2},\sigma_2)}$ are graded isomorphic as $L_k(E)$--$L_k(F)$-bimodules.
\end{prp}

\begin{proof}
We shall show that 
$$\phi \otimes \mathrm{id}_{L_k(F)} \colon Y_{(M_1,\sigma_1)} \to Y_{({M_2},\sigma_2)}$$ is   a graded bimodule isomorphism.  It is clear that $\phi \otimes \mathrm{id}_{L_k(F)}$ is a right $L_k(F)$-module isomorphism with inverse $\phi^{-1} \otimes \mathrm{id}_{L_k(F)}$.  We now  show that $\phi \otimes \mathrm{id}_{L_k(F)}$ is a left $L_k(E)$-module homomorphism. To this end, let $v \in E^0$,  $m \in M_1$, and $S \in L_k(F)$.  Then 
\begin{align*}
(\phi \otimes \mathrm{id}_{L_k(F)})( v \cdot ( m \otimes S )) &= (\phi \otimes \mathrm{id}_{L_k(F)}) ( (vm) \otimes S )  \\ 
				& = \phi(vm)\otimes S 
				= (v \phi(m)) \otimes S \\
				 & = v \cdot (\phi \otimes \mathrm{id}_{L_k(F)})(m \otimes S).
\end{align*}

Next we show that  
$$(\phi \otimes \mathrm{id}_{L_k(F)})  ( e \cdot ( m \otimes S )) = e \cdot ((\phi \otimes \mathrm{id}_{L_k(F)})( m\otimes S))$$ for all $m \in M_1$, $S\in L_k(F)$, and $e\in E^1$. Write  $\sigma_1(e\otimes m) =  \sum_{i=1}^n m_i\otimes y_i$ for $m_i \in M_1$ and $y_i \in kF^1$.  So by the explicit description of the left $L_k(E)$-action on $Y_{(M_1,\sigma_1)}$ given in Theorem \ref{thm:bimodule-action}(2), we have  $e\cdot (m\otimes S) = \sum_{i=1}^n m_i \otimes y_iS$, which then gives  
$(\phi \otimes \mathrm{id}_{L_k(F)})  ( e \cdot ( m \otimes S ))
= (\phi \otimes {\rm id}_{L_k(F)}) (\sum_{i=1}^n m_i \otimes y_iS)
= \sum_{i=1}^n \phi(m_i) \otimes y_iS$.  Applying the commuting diagram in \eqref{eq-equivalent-conj-pairs-def} to $e\otimes m$ we get 
$$\sum_{i=1}^n \phi(m_i) \otimes y_i  \ = \ \sigma_2(e\otimes \phi(m)).$$
Now to describe $e \cdot ((\phi \otimes \mathrm{id}_{L_k(F)})( m\otimes S)) = e \cdot (\phi(m)\otimes S)$, we need to calculate the expression $\sigma_2 (e \otimes \phi(m))$.   But this is precisely  $\sum_{i=1}^n \phi(m_i) \otimes y_i  $ as displayed above.   So $e \cdot ((\phi \otimes \mathrm{id}_{L_k(F)})( m\otimes S)) = e \cdot (\phi(m)\otimes S) =  \sum_{i=1}^n \phi(m_i) \otimes y_iS$ by Theorem \ref{thm:bimodule-action}(2), which is precisely $(\phi \otimes \mathrm{id}_{L_k(F)})  ( e \cdot ( m \otimes S ))$ as desired.

Further, we show 
$$(\phi \otimes \mathrm{id}_{L_k(F)})  ( e^* \cdot ( m \otimes S )) = e^* \cdot (\phi \otimes \mathrm{id}_{L_k(F)})( m\otimes S)$$ for all $m\in M_1$, $S\in L_k(F)$, and $e \in E^1$.  For each $f \in F^1$, let  $\sigma_1^{-1} (m \otimes f) = \sum_{i=1}^{n_f} y_i^f \otimes m_i^f$, where $y_1^f, \ldots, y_{n_f}^f \in kE^1$ and $m_1^f, \ldots, m_{n_f}^f \in M_1$.  By the explicit description of the left $L_k(E)$-action on $Y_{(M_1, \sigma_1)}$ given in Theorem~\ref{thm:bimodule-action}(3), $e^* \cdot (m \otimes S) =\sum_{f \in F^1} \sum_{i=1}^{n_f} (e^* y_i^f) m_i^f \otimes f^* S$.  Therefore,
$$
(\phi \otimes \mathrm{id}_{L_k(F)})(e^* \cdot (m \otimes S)) = \sum_{f \in F^1} \sum_{i=1}^{n_f} \phi((e^* y_i^f) m_i^f) \otimes f^* S = \sum_{f \in F^1} \sum_{i=1}^{n_f} (e^* y_i^f) \phi(m_i^f) \otimes f^* S,
$$
where the last equality follows from the fact that $e^* y_i^f \in kE^0$ for all $f \in F^1$ and for all $i$, and the fact that $\phi$ is a $kE^0$--$kF^0$-bimodule map.  To compute $e^* \cdot ( \phi \otimes \mathrm{id}_{L_k(F)} ) (m \otimes S) = e^*\cdot ( \phi(m) \otimes S)$, we use Theorem~\ref{thm:bimodule-action}(3) again, which requires the values of $\sigma_2^{-1}( \phi(m) \otimes f)$ for all $f \in F^1$.  By the commutative diagram in \eqref{eq-equivalent-conj-pairs-def}, $\sigma_2^{-1} =(\mathrm{id}_{kE^1} \otimes \phi) \circ \sigma_1^{-1} \circ ( \phi^{-1} \otimes \mathrm{id}_{kF^1})$.  Therefore, for all $f \in F^1$,  $\sigma_2^{-1} ( \phi(m) \otimes f) = \sum_{ i = 1}^{n_f} y_i^f \otimes \phi(m_i^f)$. Thus, by the explicit description of the left $L_k(E)$-action on $Y_{(M_2, \sigma_2)}$, we get
$e^* \cdot ( \phi(m) \otimes S) = \sum_{ f \in F^1} \sum_{ i = 1}^{n_f} (e^*y_i^f) \phi(m_i^f) \otimes f^*S$.  Hence $(\phi \otimes \mathrm{id}_{L_k(F)})(e^* \cdot (m \otimes S))  = \sum_{f \in F^1} \sum_{i=1}^{n_f} (e^* y_i^f) \phi(m_i^f) \otimes f^* S= e^*\cdot ( \phi \otimes \mathrm{id}_{L_k(F)} )(m \otimes S)$ as desired.

Since $L_k(E)$ is generated by $\{ v, e,e^* : v \in E^0, e \in E^1 \}$, we conclude that $\phi \otimes \mathrm{id}_{L_k(F)}$ respects the left $L_k(E)$-action. Thus we have established that  $\phi \otimes \mathrm{id}_{L_k(F)}$ is an $L_k(E)$--$L_k(F)$-bimodule isomorphism from $Y_{(M_1, \sigma_1)}$ to $Y_{({M_2},\sigma_2)}$.  

Finally, using the description of the grading given in Theorem~\ref{thm:bimodule-action} it is straightforward to verify $\phi \otimes \mathrm{id}_{L_k(F)}$ is a graded isomorphism.
\end{proof}

Suppose $(M, \sigma)$ is a specified conjugacy pair from $E$ to $F$ and that $(N, \psi)$ is a specified conjugacy pair from $F$ to $G$.   We wish to compose $(M, \sigma)$ and $(N, \psi)$ to form a specified conjugacy pair from $E$ to $G$.  The way to do so is straightforward but involved, so we explain it in steps.

Given a specified conjugacy pair $(M, \sigma)$ from $E$ to $F$ and a specified conjugacy pair $(N, \psi)$ from $F$ to $G$ we have isomorphisms
$$ \sigma : kE^1 \otimes_{kE^0} M \to M \otimes_{kF^0} kF^1 \qquad \text{ and } \qquad
 \psi : kF^1 \otimes_{kF^0} N \to N \otimes_{kG^0} kG^1.
$$
For the composition conjugacy it is natural to use the tensor product $M \otimes_{kF^0} N$.  Thus to obtain the specified conjugacy pair that is the composition, we seek an isomorphism from $kE^1 \otimes_{kE^0} (M \otimes_{kF^0} N)$ to $(M \otimes_{kF^0} N) \otimes_{kG^0} kG^1$.  To construct this isomorphism, we observe that we have the following sequence of $k$-vector space isomorphisms, where the three $\alpha$ maps are  the associativity isomorphisms as described in Notation~\ref{associativity-iso-notation}.
$$
\xymatrix{
kE^1 \otimes_{kE^0} (M \otimes_{kF^0} N) \ar[rr]^{\alpha_{kE^1, M, N}^{-1}} & &
(kE^1 \otimes_{kE^0} M) \otimes_{kF^0} N \ar[r]^{\sigma \otimes \operatorname{id}_{N}} & (M \otimes_{kF^0} kF^1) \otimes_{kF^0} N \ar[dlll]_(0.6){\alpha_{M, kF^1, N}} \\
M \otimes_{kF^0} (kF^1 \otimes_{kF^0} N) \ar[rr]_{ \operatorname{id}_{M} \otimes \psi} & & M \otimes_{kF^0} (N \otimes_{kG^0} kG^1) \ar[r]_{\ \alpha_{M, N, kG^1}^{-1}} & (M \otimes_{kF^0} N) \otimes_{kG^0} kG^1.
}
$$
This motivates the following definition.
\begin{dfn} \label{hash-def}
Let $(M, \sigma)$ be a specified conjugacy pair from $E$ to $F$, and let $(N, \psi)$ be a specified conjugacy pair from $F$ to $G$.  We define a $kE^0$--$kG^0$-bimodule isomorphism $$\sigma \# \psi : kE^1 \otimes_{kE^0} (M \otimes_{kF^0} N) \to (M \otimes_{kF^0} N) \otimes_{kG^0} kG^1$$ by setting 
\begin{equation} \label{hash-def-eq}
\sigma \# \psi := 
\alpha_{M, N, kG^1}^{-1}
\circ
(\operatorname{id}_{M} \otimes \psi)
\circ
\alpha_{M, kF^1, N}
\circ
(\sigma \otimes \operatorname{id}_{N})
\circ
\alpha_{kE^1, M, N}^{-1}.
\end{equation}
\end{dfn}

\begin{rmk}
In the definition of $\sigma \# \psi$ appearing in Display  \eqref{hash-def-eq} there are a number of associativity maps appearing, which make the definition ostensibly cumbersome.  If one mentally suppresses these associativity isomorphisms, then heuristically  $\sigma \# \psi$ is the composition of the map 
$$ \sigma \otimes \operatorname{id}_{N} :
(kE^1 \otimes_{kE^0} M) \otimes_{kF^0} N \to
(M \otimes_{kF^0} kF^1) \otimes_{kF^0} N $$
followed by
$$
\operatorname{id}_{M} \otimes \psi : M \otimes_{kF^0} (kF^1 \otimes_{kF^0} N) \to M \otimes_{kF^0} (N \otimes_{kG^0} kG^1).
$$
Thus one can ``think of" $\sigma \# \psi$ as being $(\operatorname{id}_{M} \otimes \psi) \circ (\sigma \otimes \operatorname{id}_{N})$ provided the proper associativity identifications are made.  However, one must keep in mind  that $(\operatorname{id}_{M} \otimes \psi) \circ ( \sigma \otimes \operatorname{id}_{N})$ is \emph{not} equal to $\sigma \otimes \psi$, and indeed this latter expression does not even make sense due to the domains and codomains of the functions involved.
\end{rmk}

\begin{prp}\label{prp-griso-tensor}
Let $(M, \sigma)$ be a specified conjugacy pair from $E$ to $F$, and let $(N, \psi)$ be a specified conjugacy pair from $F$ to $G$.  Then $$Y_{ (M \otimes_{kF^0} N, \sigma \# \psi)} \cong Y_{(M, \sigma)} \otimes_{L_k(F)} Y_{(N, \psi)}$$ as graded $L_k(E)$--$L_k(G)$-bimodules.
\end{prp}

\begin{proof}

By definition, we have bimodule isomorphisms 
$$ \sigma : kE^1 \otimes_{kE^0} M \to M \otimes_{kF^0} kF^1 \qquad \text{ and } \qquad
 \psi : kF^1 \otimes_{kF^0} N \to N \otimes_{kG^0} kG^1.
$$

\noindent
Let $$\eta \colon Y_{ (M \otimes_{kF^0} N, \sigma \# \psi)} \to Y_{(M, \sigma)} \otimes_{L_k(F)} Y_{(N, \psi)}$$ be the $k$-vector space homomorphism resulting from setting  
\begin{equation} \label{phi-iso-def}
\eta ( (m\otimes n) \otimes T) = (m \otimes 1_{L_k(F)} ) \otimes (n \otimes T)
\end{equation}
for all $m \in M$, $n \in N$, and $T \in L_k(G)$.   It is clear from the definition of the right $L_k(G)$-actions on $Y_{ (M \otimes_{kF^0} N, \sigma \# \psi)}$ and  $Y_{(M, \sigma)} \otimes_{L_k(F)} Y_{(N, \psi)}$ that $\eta$ respects the right $L_k(G)$-actions.

We now show that $\eta$ respects the left $L_k(E)$-actions.  Since $L_k(E)$ is generated by $\{ v, e, e^* : v \in E^0, e \in E^1 \}$ it is enough to show that the following three equations hold:
\begin{align*}
\eta( v \cdot ( (m\otimes n) \otimes T)) &= v \cdot \eta((m\otimes n)\otimes T), \\
\eta( e \cdot ( (m\otimes n) \otimes T)) &= e \cdot \eta((m\otimes n)\otimes T), \quad \text{ and } \\
\eta( e^* \cdot ( (m\otimes n) \otimes T)) &= e^* \cdot \eta((m\otimes n)\otimes T).
\end{align*}
Since 
\begin{align*}
\eta( v \cdot ( (m\otimes n) \otimes T) ) &=  \eta( (( vm) \otimes n) \otimes T) = ((vm) \otimes 1_{L_k(F)}) \otimes (n\otimes T) \\
& = v \cdot ((m\otimes 1_{L_k(F)}) \otimes (n\otimes T) ) =v \cdot \eta((m\otimes n)\otimes T),\\
\end{align*}
 the first equation holds.  

For the second equation, write $\sigma( e \otimes m) = \sum_{ i = 1}^K m_i \otimes y_i$ for some $m_i \in M$,  $y_i \in kF^1$, and $K \in \mathbb{N}$.  By writing each $y_i$ as a $k$-linear combination of elements in $F^1$ and moving the scalars to the first tensor factor, we may assume that each $y_i \in F^1$.  Also, for each $1\leq i \leq K$ and $n\in N$,  write $\psi( y_i \otimes n) = \sum_{j = 1}^{ K_i } n_{j, i} \otimes z_{j,i}$, where $n_{j,i} \in N$,  $z_{j,i} \in kG^1$, and $K_i\in \mathbb{N}$.  Then 
$$
(\sigma \# \psi)( e \otimes (m \otimes n)) = \sum_{ i = 1}^K \sum_{ j = 1}^{K_i} (m_i \otimes n_{j,i} ) \otimes z_{j,i}.
$$
From the description of the germane left actions on $Y_{ (M \otimes_{kF^0} N , \sigma \# \psi)}$, $Y_{(M, \sigma)}$, and $Y_{ (N, \psi)}$ given in Theorem~\ref{thm:bimodule-action}(2), we have 
\begin{align*}
e \cdot ((m \otimes n) \otimes T) &=  \sum_{ i = 1}^K \sum_{ j = 1}^{K_i} (m_i \otimes n_{j,i}) \otimes (z_{j,i}T), \\
e \cdot (m \otimes 1_{L_k(F)} ) &= \sum_{ i = 1}^K m_i \otimes y_i, \qquad \qquad \text{ and} \\
y_i \cdot (n \otimes T) &= \sum_{j = 1}^{K_i} n_{j, i} \otimes (z_{j,i}T).
\end{align*}  
Using these, along with the definition of $\eta$ given in  (\ref{phi-iso-def}), we get 
\begin{align*}
\eta(e \cdot ((m \otimes n) \otimes T)) &=  \sum_{ i = 1}^K \sum_{ j = 1}^{K_i}  ( m_i \otimes 1_{L_k(F)} ) \otimes ( n_{j, i} \otimes (z_{j,i}T))\\
&= \sum_{i=1}^K ( m_i \otimes 1_{L_k(F)} ) \otimes \left(   \sum_{ j = 1}^{K_i}   n_{j, i} \otimes (z_{j,i}T)\right) \\
&= \sum_{ i = 1}^K ( m_i \otimes 1_{L_k(F)} )\otimes ( y_i \cdot  (n \otimes T)) =  \sum_{ i = 1}^K ( m_i \otimes y_i) \otimes ( n \otimes T)\\
&= (e \cdot (m \otimes 1_{L_k(F)})) \otimes ( n \otimes T)= e \cdot ( (m \otimes 1_{L_k(F)}) \otimes (n \otimes T)) \\
&= e \cdot \eta( (m \otimes n) \otimes T),
\end{align*}
showing the second equation holds.

For the third equation,  for $n\in N$ and $g\in G^1$, write $\psi^{-1} (n \otimes g) = \sum_{i=1}^{K_g} y_i^g \otimes n_i^g$, where $y_i^g \in kF^1$, $n_i^g \in N$, and $K_g \in \mathbb{N}$.  By writing each $y_i^g$ as a $k$-linear combination of elements in $F^1$ and moving the scalars to the second tensor factor, we may assume each $y_i^g \in F^1$.  Thus 
$$\psi^{-1} (n \otimes g) = \sum_{i=1}^{K_g} y_i^g \otimes n_i^g =  \sum_{i=1}^{K_g} y_i^g r(y_i^g) \otimes n_i^g = \sum_{i=1}^{K_g} y_i^g  \otimes r(y_i^g) n_i^g.$$  For each $f \in F^1$, write $\sigma^{-1}(m \otimes f) = \sum_{i=1}^{K_f} x_i^f \otimes m_i^f$ for some $x_i^f \in kE^1$, $m_i^f \in M$, and $K_f \in \mathbb{N}$.  Then 
$$
(\sigma \# \psi)^{-1} ( (m\otimes n) \otimes g) = \sum_{i = 1}^{K_g} \sum_{ j=1}^{K_{y_i^g}} x_{j}^{ y_i^g} \otimes ( m_j^{ y_i^g} \otimes r( y_i^g) n_i^g).  
$$
From the description of the germane left actions on $Y_{ (M \otimes_{kF^0} N , \sigma \# \psi)}$, $Y_{(M, \sigma)}$, and $Y_{ (N, \psi)}$ given in Theorem~\ref{thm:bimodule-action}(3), we have
\begin{align}
e^*\cdot ( (m\otimes n) \otimes T) &= \sum_{ g \in G^1}  \sum_{i = 1}^{K_g} \sum_{ j=1}^{K_{y_i^g}}  ( (e^*x_{j}^{ y_i^g}) m_j^{ y_i^g} \otimes r( y_i^g) n_i^g) \otimes g^*T \label{e*-three-tensor-eq} \\
e^*\cdot (m \otimes 1_{L_k(F)}) &= \sum_{ f \in F^1} \sum_{ j = 1}^{K_f} ((e^*x_{j}^{ f}) m_j^{ f}) \otimes f^*, \qquad \text{ and } \label{e*-m-1-eq} \\
f^*\cdot (n \otimes T) &= \sum_{ g \in G^1}  \sum_{i = 1}^{K_g} ((f^*y_i^g) n_i^g )\otimes g^*T. \label{f*-n-tensor-T-eq}
\end{align}  
Therefore 
\begin{align*}
&\  \eta ( e^* \cdot ( (m\otimes n) \otimes T) ) \\
&= \sum_{ g \in G^1}  \sum_{i = 1}^{K_g} \sum_{ j=1}^{K_{y_i^g}}  ( ((e^*x_{j}^{ y_i^g}) m_j^{ y_i^g}) \otimes 1_{L_k(F)}) \otimes ( r( y_i^g) n_i^g \otimes g^*T) \quad \text{(by \eqref{e*-three-tensor-eq} and the definition of $\eta$)} \\
&= \sum_{ g \in G^1}  \sum_{i = 1}^{K_g} \sum_{ j=1}^{K_f}  \sum_{f \in F^1} ( ((e^*x_{j}^{ f}) m_j^{ f}) \otimes 1_{L_k(F)}) \otimes ( ((f^*y_i^g) n_i^g )\otimes g^*T) \  \text{(by (CK1): $f^*y_i^g = \delta_{f,y_i^g} r(y_i^g)$)} \\
&= \sum_{ g \in G^1}  \sum_{i = 1}^{K_g} \sum_{f \in F^1} \sum_{ j=1}^{K_f}  ( ((e^*x_{j}^{ f}) m_j^{ f}) \otimes 1_{L_k(F)}) \otimes ( ((f^*y_i^g) n_i^g )\otimes g^*T)\\
&= \sum_{ f \in F^1} \sum_{ j = 1}^{K_f} ( ((e^*x_{j}^{ f}) m_j^{ f}) \otimes 1_{L_k(F)}) \otimes \left(   \sum_{ g \in G^1}  \sum_{i = 1}^{K_g} ((f^*y_i^g) n_i^g )\otimes g^*T\right)  \\
&=  \sum_{ f \in F^1} \sum_{ j = 1}^{K_f} ( ((e^*x_{j}^{ f}) m_j^{ f}) \otimes 1_{L_k(F)}) \otimes ( f^* \cdot ( n \otimes T)) \qquad \text{(by \eqref{f*-n-tensor-T-eq})} \\
&= \sum_{ f \in F^1} \sum_{ j = 1}^{K_f} ( ((e^*x_{j}^{ f}) m_j^{ f}) \otimes f^*) \otimes (n \otimes T) \qquad \text{(since the middle tensor product is over $L_k(F)$)}  \\
&= ( e^* \cdot ( m \otimes 1_{L_k(F)}))  \otimes (n \otimes T)  \qquad \text{(by \eqref{e*-m-1-eq})} \\
&= e^*\cdot  (( m \otimes 1_{L_k(F)})  \otimes (n \otimes T) )\\
& = e^*\cdot \eta ((m\otimes n) \otimes T) \qquad \text{(by \eqref{phi-iso-def})}, 
\end{align*}
and the third equation holds.  We may thus conclude that $\eta$ respects the left $L_k(E)$-actions.  Consequently $\eta$ is a $L_k(E)$--$L_k(G)$-bimodule homomorphism.

We now  prove that $\eta$ is in fact an isomorphism.  Let $\mu \colon L_k(F) \otimes_{L_k(F)} Y_{(N, \psi)} \to Y_{(N, \psi)}$ be the right $L_k(G)$-module isomomorphism given by $\mu( S \otimes (n \otimes T) ) = S \cdot (n \otimes T)$.  Let $$\tau \colon Y_{(M , \sigma) } \otimes_{L_k(F)} Y_{(N, \psi)} \to Y_{ (M \otimes_{kF^0} N, \sigma \# \psi)}$$ be the composition $\alpha_{ M , N, L_k(G)}^{-1} \circ ( \mathrm{id}_M \otimes \mu ) \circ \alpha_{ M , L_k(F), N \otimes_{kG^0} L_k(G) }$.  On elementary tensors, $\tau$ sends $( m \otimes S) \otimes (n \otimes T)$ to $\sum_{ i = 1}^K (m \otimes n_i) \otimes T_i$, where $S \cdot (n \otimes T) = \sum_{ i = 1}^K n_i \otimes T_i$.  Since $1_{L_k(F)} \cdot ( n \otimes T) = n \otimes T$, 
\begin{align*}
(\tau \circ \eta)( (m\otimes n) \otimes T) &= \tau ( (m \otimes 1_{L_k(F)})\otimes (n \otimes T)) = (m \otimes n) \otimes T \quad \text{and}\\
(\eta \circ \tau) ( ( m \otimes S) \otimes (n \otimes T)) &= \sum_{ i = 1}^K (m \otimes 1_{L_k(F)} )\otimes (n_i \otimes T_i) =  (m \otimes 1_{L_k(F)} )\otimes  ( S \cdot (n \otimes T)) \\
&= ( m \otimes S ) \otimes (n \otimes T).
\end{align*}
Therefore, $\eta$ is an isomorphism.

Using the gradings on $Y_{(M, \sigma)}$ and $Y_{(N, \psi)}$ given in Theorem~\ref{thm:bimodule-action}, $m \otimes 1_{L_k(F)}$ is in the zeroth component of $Y_{(M, \sigma)}$ and $n\otimes T \in (Y_{(N, \psi)})_\gamma$ for all $\gamma \in \ZZ$ and all $T \in L_k(G)_\gamma$.  Consequently, $(m \otimes 1_{L_k(F)}) \otimes (n \otimes T) \in ( Y_{(M, \sigma)} \otimes Y_{(N, \psi)})_\gamma$ for all $\gamma \in \ZZ$ and all $T \in L_k(G)_\gamma$.  Hence we have that $\eta$ is a graded isomorphism.
\end{proof}

\begin{prp} \label{hash-comp-well-defined-prop}
Let $(M_1, \sigma_1)$ and $(M_2, \sigma_2)$ be specified conjugacy pairs from $E$ to $F$, and let $(N_1, \psi_1)$ and $(N_2, \psi_2)$ be specified conjugacy pairs from $F$ to $G$.  If $(M_1, \sigma_1)$ is equivalent to $(M_2, \sigma_2)$ and  $(N_1,\psi_1)$ is equivalent to $(N_2, \psi_2)$, then the composition $(M_1 \otimes_{kF^0} N_1, \sigma_1 \# \psi_1)$ is equivalent to the composition $(M_2 \otimes_{kF^0} N_2, \sigma_2 \# \psi_2)$.
\end{prp}

\begin{proof}
Since $(M_1, \sigma_1)$ is equivalent to $(M_2, \sigma_2)$ and  $(N_1, \psi_1)$ is equivalent to $(N_2, \psi_2)$ there exist bimodule isomorphisms $\phi : M_1 \to M_2$ and $\xi : N_1 \to N_2$ making the following diagrams commute:
$$
\vcenter{
\xymatrix{
kE^1 \otimes_{ kE^0 } M_1 \ar[r]^-{\sigma_1} \ar[d]_{\mathrm{id}_{kE^1} \otimes \phi }& M_1 \otimes_{ kF^0 } kF^1 \ar[d]^{\phi \otimes \mathrm{id}_{kF^1} } \\
kE^1 \otimes_{ kE^0 } {M_2}  \ar[r]_-{\sigma_2} & {M_2} \otimes_{ kF^0 } kF^1
}
}
\qquad \text{ and } \qquad
\vcenter{
\xymatrix{
kF^1 \otimes_{ kF^0 } N_1 \ar[r]^-{\psi_1} \ar[d]_{\mathrm{id}_{kF^1} \otimes \xi }& N_1 \otimes_{ kG^0 } kG^1 \ar[d]^{\xi \otimes \mathrm{id}_{kG^1} } \\
kF^1 \otimes_{ kF^0 } {N_2}  \ar[r]_-{\psi_2} & {N_2} \otimes_{ kG^0 } kG^1.
}
}
$$
It is straightforward, albeit tedious, to show that these facts imply the diagram
$$
\xymatrix{
kE^1 \otimes_{ kE^0 } (M_1 \otimes_{kF^0} N_1) \ar[rr]^-{\sigma_1 \# \psi_1} \ar[d]_{\mathrm{id}_{kE^1} \otimes (\phi \otimes \xi) }& & (M_1 \otimes_{kF^0} N_1) \otimes_{ kG^0 } kG^1 \ar[d]^{(\phi \otimes \xi) \otimes \mathrm{id}_{kG^1} } \\
kE^1 \otimes_{ kE^0 } (M_2 \otimes_{kF^0} N_2)  \ar[rr]_-{\sigma_2 \# \psi_2} & & (M_2 \otimes_{ kF^0 } N_2) \otimes_{kG^0} kG^1
}
$$
commutes.  Thus  $(M_1 \otimes_{kF^0} N_1, \sigma_1 \# \psi_1)$ is equivalent to $(M_2 \otimes_{kF^0} N_2, \sigma_2 \# \psi_2)$.
\end{proof}

The following result will be used in  Remark \ref{categoricalperspectiveremark} to help provide a  categorical interpretation of the bridging bimodule.  

\begin{prp} \label{identity-conj-pair-prop}
Let $\operatorname{id}_E : E \to E$  be the specified conjugacy pair $(kE^0, \epsilon_E)$ described in Example \ref{identityconjpair}.   
\begin{itemize}
\item[(1)] Let $E$ and $F$ be finite graphs with no sinks.  If $(M, \sigma)$ is a specified conjugacy pair from $E$ to $F$, then $(M, \sigma) \circ \operatorname{id}_E$ is equivalent to $(M, \sigma)$ and $\operatorname{id}_F \circ (M, \sigma)$ is equivalent to $(M, \sigma)$.
\item[(2)] If $E$ is any finite graph with no sinks, the bridging bimodule $Y_{\operatorname{id}_E}$ is graded isomorphic to $L_k(E)$.
\end{itemize}
\end{prp}

\begin{proof}
For (1) let $\phi: kE^0 \otimes_{kE^0} M \to M$ be the isomorphism with $\phi (v \otimes m) = vm$.  One can  easily verify that the diagram 
$$
\xymatrix{
kE^1 \otimes_{ kE^0 } (kE^0 \otimes_{kE^0} M) \ar[rr]^-{\epsilon_E \# \sigma} \ar[d]_{\mathrm{id}_{kE^1} \otimes \phi } & & (kE^0 \otimes_{kE^0} M) \otimes_{ kF^0 } kF^1 \ar[d]^{\phi \otimes \mathrm{id}_{kF^1} } \\
kE^1 \otimes_{ kE^0 } {M}  \ar[rr]_-{\sigma} & & {M} \otimes_{ kF^0 } kF^1
}
$$
commutes.  Thus $(kE^0 \otimes_{kE^0} M, \epsilon_E \# \sigma) = (M, \sigma) \circ \operatorname{id}_E$ is equivalent to $(M, \sigma)$.  A similar argument shows $\operatorname{id}_F \circ (M, \sigma)$ is equivalent to $(M, \sigma)$.

For (2), recall that $Y_{\operatorname{id}_E} := kE^0 \otimes_{kE^0}  L_k(E)$ with the natural right $L_k(E)$-action and with the left $L_k(E)$-action and $\mathbb{Z}$-grading determined by $\epsilon_E$  as described in Theorem~\ref{thm:bimodule-action}.  Define $\Psi : kE^0 \otimes_{kE^0} L_k(E) \to L_k(E)$ to be the unique linear map with 
$$\Psi ( w \otimes S) := wS \qquad  \text{for $w \in E^0$ and $S \in L_k(E)$.}$$
It is straightforward to show $\Psi$ is a $k$-linear isomorphism that preserves the right $L_k(E)$-action and the $\mathbb{Z}$-grading.  It remains to show that $\Psi$ also preserves the left $L_k(E)$-action.  To this end, let $v,w \in E^0$ and $S \in L_k(E)$.  Using the description of the left $L_k(E)$-action given in Theorem~\ref{thm:bimodule-action} we have $v \cdot (w \otimes S) := vw \otimes S$.  Thus
\begin{equation} \label{epsilon-one-eq}
\Psi ( v \cdot (w \otimes S)) = \Psi(vw \otimes S) = vwS = v \Psi(w \otimes S).
\end{equation}
In addition, if we let $e \in E^1$, $w \in E^0$, and $S \in L_k(E)$, then $e \otimes w \in kE^1 \otimes_{kE^0} kE^0$ is nonzero if and only if $w = r(e)$, in which case $\epsilon_E (e \otimes w) = \epsilon_E (e \otimes r(e)) = s(e) \otimes e$, and hence by Theorem~\ref{thm:bimodule-action} $e \cdot (w \otimes S) = s(e) \otimes eS$.  Thus in the nonzero case (i.e., when $w = r(e)$) we have
\begin{equation} \label{epsilon-two-eq}
\Psi ( e \cdot (w \otimes S)) = \Psi(s(e) \otimes eS) = s(e)eS = eS=er(e)S =ewS= e\Psi(w \otimes S).
\end{equation}
Finally, if we let $e \in E^1$, $w \in E^0$, and $S \in L_k(E)$, then for any $f \in E^1$ we have $w \otimes f$ is nonzero if and only if $w = s(f)$, in which case $\epsilon_E^{-1} (w \otimes f) = \epsilon_E^{-1} (s(f) \otimes f) = f \otimes r(f)$, and hence by Theorem~\ref{thm:bimodule-action} we have 
$$e^* \cdot (w \otimes S) = \sum_{f \in E^1} e^*f r(f) \otimes f^*S = e^*e \otimes e^*S.$$
Thus in the nonzero case (i.e., when $w = s(e)$) we have
\begin{equation} \label{epsilon-three-eq}
\Psi ( e^* \cdot (w \otimes S)) = \Psi( e^*e \otimes e^*S) = e^*ee^*S =e^* S =e^*s(e)S= e^*wS = e^* \Psi (w \otimes S).
\end{equation}
Since $\{v,e,e^* : v \in E^0, e \in E^1\}$ generates $L_k(E)$ and $\{ w : w \in E^0 \}$ generates $kE^0$, Equations \eqref{epsilon-one-eq}, \eqref{epsilon-two-eq}, and \eqref{epsilon-three-eq} imply $\Psi$ preserves the left $L_k(E)$-action.  Thus $\Psi$ is a graded $L_k(E)$--$L_k(E)$-bimodule isomorphism from $Y_{\operatorname{id}_E} := kE^0 \otimes L_k(E)$ onto $L_k(E)$.
\end{proof}

\begin{rmk}\label{categoricalperspectiveremark}
We conclude this section by pointing out a categorical perspective that allows us to view the bridging bimodule construction as a functor.  Let $\mathsf{Conj}$ denote the category whose objects are finite directed graphs with no sinks and whose morphisms are equivalence classes of specified conjugacy pairs (as defined in Definition~\ref{conjugacy-def} and Definition~\ref{equivalent-conj-pairs-def}), with composition of equivalence classes defined by setting 
$$ [(N, \psi)] \circ [(M, \sigma)] := [( M \otimes_{kF^0} N, \sigma \# \psi)]$$
for $[(M, \sigma)] : E \to F$ and $[(N, \psi)] : F \to G$.   The fact this composition is well defined is precisely the result of Proposition~\ref{hash-comp-well-defined-prop}.  In addition, one can verify that associativity of composition holds.  (Indeed, this is why we describe equivalence classes of specified conjugacy pairs, as associativity of $\#$ holds up to equivalence, but not necessarily up to equality.)  Furthermore, Proposition~\ref{identity-conj-pair-prop}(1) shows that for any directed graph $E$ the equivalence class of the specified conjugacy pair $\operatorname{id}_E := ( kE^0, \epsilon_E)$ is the identity morphism on $E$.

Let $\mathsf{Gr}$-$\mathsf{BiMod}$ denote the category whose objects are $\mathbb{Z}$-graded (unital) rings and whose morphisms are graded isomorphism classes of $\mathbb{Z}$-graded bimodules,  with composition given by tensor product.  If $M$ is a  $\mathbb{Z}$-graded $R$--$S$-bimodule and $N$ is a $\mathbb{Z}$-graded $S$--$T$-bimodule, the $\mathbb{Z}$-grading on the tensor product $M \otimes_S N$ is defined to have $\ell$th-component
$$(M \otimes_S N)_\ell := \Big\{ \sum_i m_i \otimes n_i   \ | \ \operatorname{deg}_M(m_i) + \operatorname{deg}_N(n_i) = \ell \Big\},$$
where $m_i \in M$ and $n_i \in N$ are homogeneous elements of $M$ and $N$, respectively.

For each field $k$ one may define a functor from $\mathsf{Conj}$ to $\mathsf{Gr}$-$\mathsf{BiMod}$ as follows.  On objects the functor takes a finite directed graph with no sinks $E$ to the Leavitt path algebra $L_k(E)$ viewed as a $\mathbb{Z}$-graded ring.  On morphisms the functor takes an equivalence class of a specified conjugacy pair $[(M, \sigma)]$ to the isomorphism class of the bridging bimodule $[Y_{(M, \sigma)}]$.  This assignment is well-defined by Proposition~\ref{prop:MUF-iso}.  In addition, Proposition~\ref{prp-griso-tensor} shows that the functor preserves composition; i.e., 
$$Y_{ (M \otimes_{kF^0} N, \sigma \# \psi)} \cong Y_{(M, \sigma)} \otimes_{L_k(F)} Y_{(N, \psi)}$$
is precisely the statement
$$ [ Y_{  (N, \psi) \circ (M, \sigma)  } ] = [Y_{ (N, \psi) } ] \circ [Y_{ (M, \sigma) } ].$$
Finally, Proposition~\ref{identity-conj-pair-prop}(2) shows that the functor preserves identity morphisms; i.e., $[Y_{\operatorname{id}_E} ] = [L_k(E)]$. 
\end{rmk}

\section{Conditions for shift equivalence to imply graded Morita equivalence} \label{Theorem-B-sec}

Prior to establishing Theorem~\ref{thm:main2}, we need the following key property of a specific bridging bimodule.   Let $E$ be a graph with no sinks and let $m\in \NN$.  We define 
$$\nu_m^E:  \  kE^1 \otimes_{kE^0} (kE^1)^{\otimes m} \to (kE^1)^{\otimes m} \otimes_{kE^0} kE^1, $$
$$x_1 \otimes (y_1 \otimes \cdots \otimes y_m) \mapsto (x_1 \otimes y_1 \otimes  \cdots \otimes y_{m-1} ) \otimes y_m.$$ 
The map $\nu_m^E$ is easily seen to be a $kE^0$--$kE^0$-bimodule isomorphism.    
Noting that $(kE^1)^{\otimes m}$ is a $kE^0$--$kE^0$-bimodule, Theorem~\ref{thm:bimodule-action} yields that the bridging bimodule 
 $$ Y_{ ((kE^1)^{\otimes m} , \nu_m^E)}   \ := \  (kE^1)^{\otimes m} \otimes_{kE^0} L_k(E)$$ 
is an $L_k(E)$--$L_k(E)$-bimodule.

\begin{lem}\label{lem:npowergraph2}
Let $E$ be a finite graph with no sinks and let $m\in \NN$.  Then
$$Y_{ ((kE^1)^{\otimes m} , \nu_m^E)} \cong L_k(E)$$ 
as $L_k(E)$--$L_k(E)$-bimodules.
\end{lem}

\begin{proof}
Let $$\rho_{E, m} \colon Y_{ ((kE^1)^{\otimes m} , \nu_m^E)} \to L_k(E)$$ be the $kE^0$ -- $L_k(E)$-bimodule isomorphism given in  Proposition~\ref{lem:iso-YotimesOY}.  It suffices to prove that $\rho_{E,m}$ preserves the left $L_k(E)$-module structure.  Using the left $L_k(E)$-module structure on $Y_{ ((kE^1)^{\otimes m} , \nu_m^E)}$ described in Theorem~\ref{thm:bimodule-action}, for any $v \in E^0$ and $e \in E^1$ we have
\begin{align*}
    \rho_{E, m} ( v \cdot ( x_1 \otimes \cdots \otimes x_m) \otimes S ) &= \rho_{E,m} ( (vx_1) \otimes x_2 \otimes \cdots \otimes x_m \otimes S )= (vx_1) x_2 \cdots x_m S \\
    &= v( x_1 \cdots x_m S)
     = v \cdot \rho_{E, m} ( x_1 \otimes \cdots \otimes x_m \otimes S ), \\
    \rho_{E,m} ( e \cdot ( x_1 \otimes \cdots \otimes x_m )\otimes S)  &= \rho_{E,m} ( (e \otimes x_1 \otimes \cdots \otimes x_{m-1}) \otimes x_mS) 
                    \\
                    &= e (x_1 \cdots x_m S)   = e\cdot  \rho_{E,m} ( x_1 \otimes \cdots \otimes x_m \otimes S), \mbox{ and} \\
    \rho_{E,m}( e^*\cdot  ( x_1 \otimes \cdots \otimes x_m) \otimes S) &= \rho_{E, m} \left( e^*\cdot  (( x_1 \otimes \cdots \otimes x_m) \otimes  \sum_{f \in E^1}   ff^*S) \right) \\
     &= \rho_{E, m} \left( \sum_{f \in E^1}  e^*\cdot  (( x_1 \otimes \cdots \otimes x_m) \otimes ff^*S) \right) \\
    &= \sum_{ f \in E^1} \rho_{E, m} (  (e^*x_1)(x_2 \otimes \cdots \otimes x_{m} \otimes f ) \otimes f^*S ) \\
    &= \sum_{ f \in E^1}  e^*x_1 \rho_{E, m}( (x_2 \otimes \cdots \otimes x_{m} \otimes f ) \otimes f^*S) \\
    &= \sum_{ f \in E^1}  e^*x_1 x_2 \cdots x_m ff^*S  \\
    &= e^*x_1 x_2 \cdots x_m   \sum_{ f \in E^1}   ff^*S  \\
    &= e^*x_1 x_2 \cdots x_m S \\
    &= e^* \cdot \rho_{E,m}( x_1 \otimes \cdots \otimes x_m \otimes S).
\end{align*}
Thus $\rho_{E,m}$ preserves the left $L_k(E)$-module structure, and thus $\rho_{E,m}$ is an $L_k(E)$--$L_k(E)$-bimodule isomorphism.
\end{proof}

\begin{rmk}\label{nu1Eremark}
We note that the $L_k(E)$--$L_k(E)$-bimodule isomorphism $\rho_{E,m}$ of Lemma \ref{lem:npowergraph2} is not a $\mathbb{Z}$-graded isomorphism.  However, $\rho_{E,m}$ does give a graded isomorphism between $Y_{ ((kE^1)^{\otimes m} , \nu_m^E)} $ and the $m$-suspension bimodule $L_k(E)(m)$ described in the following subsection. 

Additionally, we highlight the fact  that $\nu_1^E: kE^1 \otimes kE^1 \to kE^1 \otimes kE^1$ is the identity map ${\rm id}_{kE^1 \otimes kE^1}$.  This trivial observation turns out to  play a key role in the proof of Proposition~\ref{SE-implies-COM-prop} below.  
\end{rmk}

\subsection{Graded Morita equivalence, and our second main result}\label{subsectiongMe}
We are now in position to prove  Theorem~\ref{thm:main2}, which establishes that if $E$ and $F$ are finite graphs with no sinks, the existence of a certain pair of commutative diagrams involving the  maps $\nu_m^E$ and $\nu_n^F$ implies that the Leavitt path algebras $L_k(E)$ and $L_k(F)$ are graded Morita equivalent.

In particular, this gives the following.  Suppose the adjacency matrices of $E$ and $F$ are shift equivalent, which by Theorem~\ref{thm:shift-equivalence-module} is equivalent to the existence of four associated  bimodule isomorphisms.  If these four bimodule isomorphisms can in addition be chosen to satisfy the two commutative diagrams presented in the statement of  Theorem~\ref{thm:main2}, then $L_k(E)$ and $L_k(F)$ are graded Morita equivalent.  In other words, shift equivalence of the adjacency matrices of $E$ and $F$, when accompanied by an appropriate commutativity condition on associated isomorphisms, implies  $L_k(E)$ and $L_k(F)$ are graded Morita equivalent. 

Before stating and proving Theorem~\ref{thm:main2} we recall the key ideas regarding  {\it graded Morita equivalence} for rings.  (See \cite[\S2.3]{RoozbehBook} for additional information.)  While the notation we use here is standard in the literature, it can seem somewhat ill-chosen on first encounter.  

For a ring $R$, $\mathsf{Mod}$-$R$ denotes the category of right $R$-modules.  For an abelian group $(\mathcal{G},+)$, and a $\mathcal{G}$-graded ring $R$,  $\mathsf{Gr}$-$R$ denotes the category whose objects are $\mathcal{G}$-graded right $R$-modules, and morphisms are $\mathcal{G}$-graded $R$-homomorphisms.   We denote the $\mathcal{G}$-decomposition of $M$ by $\oplus_{g\in \mathcal{G}}M_g$;  if $0 \neq m\in M$ has $m\in M_h$ for some $h\in \mathcal{G}$, we say $m$ is {\it homogeneous}, and write ${\rm deg}_M(m) = h$.   

For an abelian group $\mathcal{G}$, a $\mathcal{G}$-graded ring $R$, a $\mathcal{G}$-graded right $R$-module $M$, and $\gamma \in \mathcal{G}$, the $\gamma${\it -suspension} of $M$, denoted $M(\gamma)$, is defined to be the $\mathcal{G}$-graded right $R$-module having $M(\gamma) = M$ as right $R$-modules, but with $\mathcal{G}$-grading defined by setting
$$M(\gamma)_g = M_{\gamma +g}$$
for  each $g\in \mathcal{G}$.   So for each homogeneous element $0\neq m\in M$,  ${\rm deg}_{M(\gamma)}(m) = {\rm deg}_M(m)~-~\gamma$.       
For $\gamma \in \mathcal{G}$, the functor 
$$\Psi_\gamma: \mathsf{Gr}\mbox{-}R \to \mathsf{Gr}\mbox{-}R$$
 is defined by setting $\Psi_\gamma (M) = M(\gamma)$ on objects (and the identity on morphisms).  
 
 For any $\mathcal{G}$-graded ring $R$, the {\it forgetful  functor} 
$$U_R:    \mathsf{Gr}\mbox{-}R \to \mathsf{Mod}\mbox{-}R$$  is the identity on both objects and morphisms, but views each as their ungraded counterpart.

 \begin{dfn}\label{gradedfunctordef}
   Let $R$ and $S$ be  $\mathcal{G}$-graded rings.  
$$\mbox{A functor } \Phi: \mathsf{Gr}\mbox{-}R \to \mathsf{Gr}\mbox{-}S \mbox{ is called a  {\it  graded functor}
 }$$
in case $ \Psi_\gamma( \Phi (M)) = \Phi ( \Psi_\gamma (M))$ as $\mathcal{G}$-graded right $R$-modules, for each $\mathcal{G}$-graded right $R$-module $M$ and each $\gamma \in \mathcal{G}$.   
$$\mbox{A functor } \varphi: \mathsf{Mod}\mbox{-}R \to \mathsf{Mod}\mbox{-}S \mbox{ is called a  {\it  graded functor}
 }$$
 in case there is a graded functor $\Phi: \mathsf{Gr}$-$R$ $\to$ $\mathsf{Gr}$-$S$ for which $U_S \circ \Phi = \varphi \circ U_R$ as functors from $\mathsf{Gr}$-$R$ to $\mathsf{Mod}\mbox{-}S$. 
\end{dfn}

\noindent 
We emphasize that the phrase ``graded functor" is used in two different contexts in the previous definition.  

\begin{exm}\label{tensorfunctorexample}
Let $R$ and $S$ be $\mathcal{G}$-graded rings, and let ${}_RX_S$ be a $\mathcal{G}$-graded $R$--$S$-bimodule.     We define the functor 
$\Phi_X : \mathsf{Gr}$-$R$ $\to$ $\mathsf{Gr}$-$S$  by setting 
$$\Phi_X(M_R) = M \otimes_R X$$
for each $\mathcal{G}$-graded right $R$-module $M$, and setting $\Phi_X(f) = f \otimes {\rm id}_X$ for each morphism $f: M \to N$ in  $\mathsf{Gr}$-$R$.   

Then $\Phi$ is a graded functor from $\mathsf{Gr}$-$R$ $\to$ $\mathsf{Gr}$-$S$.  To verify this, we check that  $ \Psi_\gamma( \Phi_X (M)) = \Phi_X ( \Psi_\gamma (M))$ as $\mathcal{G}$-graded right $R$-modules; i.e., that $(M \otimes_R X)(\gamma) = M \otimes_R (X(\gamma))$ as $\mathcal{G}$-graded modules for all $\gamma \in \mathcal{G}$.   Recalling the definition of the grading on tensor products given in  Remark \ref{categoricalperspectiveremark}  and  the $\mathcal{G}$-grading on the $\gamma$-suspension given directly above,   we have that for each $\ell \in \mathcal{G}$, $$((M\otimes_R X)(\gamma))_\ell = \{ \sum m_i \otimes x_i \ | \ {\rm deg}_M(m_i) + {\rm deg}_X(x_i) = \ell + \gamma\}.$$  On the other hand,  
\begin{eqnarray*} (M \otimes_R (X(\gamma)))_\ell & = &  \{ \sum m_i \otimes x_i \ | \ {\rm deg}_M(m_i) + {\rm deg}_{X(\gamma)}(x_i) =  \ell\}  \\ 
& = & \{ \sum m_i \otimes x_i \ | \ {\rm deg}_M(m_i) + ( {\rm deg}_{X}(x_i) - \gamma ) = \ell\},
\end{eqnarray*}
which then clearly equals $((M\otimes X)(\gamma))_\ell$ by the first display.

We now define the (different, but obviously related) functor $\varphi_X: \mathsf{Mod}$-$R$ $\to$ $\mathsf{Mod}$-$S$  by setting $$\varphi_X(M_R) = M \otimes_R X$$
for each right $R$-module $M$, and setting $\varphi_X(f) = f \otimes {\rm id}_X$ for each morphism $f: M \to N$ in  $\mathsf{Mod}$-$R$.   

Then $\varphi_X: \mathsf{Mod}$-$R$ $\to$ $\mathsf{Mod}$-$S$ is a graded functor from $\mathsf{Mod}$-$R$ to $\mathsf{Mod}$-$S$, because clearly the previously described graded functor $\Phi_X: $    $\mathsf{Gr}$-$R$ $\to$ $\mathsf{Gr}$-$S$ satisfies  $U_S \circ \Phi_X = \varphi_X \circ U_R$ as functors from  $\mathsf{Gr}$-$R$ $\to$ $\mathsf{Mod}$-$S$.  

\end{exm}

\begin{dfn} Let $R$ and $S$ be $\mathcal{G}$-graded rings.
  $$\mbox{A graded functor } \Phi: \mathsf{Gr}\mbox{-}R \to \mathsf{Gr}\mbox{-}S \mbox{ is called a {\it graded equivalence}
  }$$
 in case there is a graded functor $\Gamma: \mathsf{Gr}$-$S$ $\to$ $\mathsf{Gr}$-$R$ such that $\Gamma\circ \Phi$ and $\Phi\circ \Gamma$ are naturally isomorphic to the identity functors on $\mathsf{Gr}$-$R$ and $\mathsf{Gr}$-$S$, respectively.  
$$\mbox{A graded functor } \varphi: \mathsf{Mod}\mbox{-}R \to \mathsf{Mod}\mbox{-}S \mbox{ is called a {\it graded equivalence}
  }$$
 in case  $\varphi$ is an equivalence from $\mathsf{Mod}\mbox{-}R \to \mathsf{Mod}\mbox{-}S$ in the usual sense;  that is, in case  there is a (not necessarily graded) functor $\tau: \mathsf{Mod}$-$S$ $\to$ $\mathsf{Mod}$-$R$ such that $\tau \circ \varphi$ and $\varphi \circ \tau$ are naturally isomorphic to the identity functors on $\mathsf{Mod}$-$R$ and $\mathsf{Mod}$-$S$, respectively.  
\end{dfn}

An obvious upshot of the previous two definitions is that there are two settings in which the phrase {\it graded equivalence} is used:   one  as a condition on a functor  from $\mathsf{Gr}$-$R$ to $\mathsf{Gr}$-$S$,   the other as a condition on a functor  from $\mathsf{Mod}$-$R$ to   $\mathsf{Mod}$-$S$.   By a powerful (perhaps surprising) theorem of Hazrat, there is no  ambiguity in doing so.   

{\bf Graded Morita Equivalence Theorem.}  (\cite[(1)$\Leftrightarrow$(2) of Theorem~2.3.8]{RoozbehBook})   For $\mathcal{G}$-graded rings $R$ and $S$, there exists a graded equivalence from   $\mathsf{Gr}\mbox{-}R \mbox{ to } \mathsf{Gr}\mbox{-}S$ if and only if there exists a graded  equivalence from  $\mathsf{Mod}\mbox{-}R \mbox{ to } \mathsf{Mod}\mbox{-}S$.

\smallskip

The Graded Morita Equivalence Theorem allows for the following.

\begin{dfn}\label{gradedequivalentdef}
Let $R$ and $S$ be $\mathcal{G}$-graded rings.  We say that $R$ and $S$ are  {\it graded Morita equivalent} (or more succinctly  {\it graded equivalent}) in case either one of these two equivalent conditions holds:

(geMod)  \ \  there exists a graded equivalence from $\mathsf{Mod}\mbox{-}R$ to $ \mathsf{Mod}\mbox{-}S$, or

(geGr)  \ \ \  \    there exists a graded equivalence from $\mathsf{Gr}\mbox{-}R$ to $ \mathsf{Gr}\mbox{-}S$.

\end{dfn}

\smallskip

For the statement of Theorem~\ref{thm:main2}, we remind the reader of the description of the $\#$ operation given in Definition~\ref{hash-def}.  Let $(M, \sigma)$ be a specified conjugacy pair from $E$ to $F$, and let $(N, \psi)$ be a specified conjugacy pair from $F$ to $G$.  We define the $kE^0$--$kG^0$-bimodule isomorphism $\sigma \# \psi : kE^1 \otimes_{kE^0} (M \otimes_{kF^0} N) \to (M \otimes_{kF^0} N) \otimes_{kG^0} kG^1$ by setting
$$
\sigma \# \psi := 
\alpha_{M, N, kG^1}^{-1}
\circ
(\operatorname{id}_{M} \otimes \psi)
\circ
\alpha_{M, kF^1, N}
\circ
(\sigma \otimes \operatorname{id}_{N})
\circ
\alpha_{kE^1, M, N}^{-1}.
$$

\begin{thm}[Sufficient conditions for shift equivalence to imply graded Morita equivalence]\label{thm:main2}
Let $k$ be any field, and let $E$ and $F$ be finite graphs with no sinks.  Suppose there exists a $kE^0$--$kF^0$-bimodule $M$, a $kF^0$--$kE^0$-bimodule $N$, and a positive integer $n$ for which there exist bimodule isomorphisms 
\begin{align*}
\omega_E \colon M \otimes_{kF^0} N &\to {(kE^1})^{\otimes n} & \omega_{F} \colon N \otimes_{kE^0} M &\to ({kF^1})^{\otimes n} \\
\sigma_M \colon {kE^1} \otimes_{kE^0} M &\to M \otimes_{kF^0} {kF^1} & \text{  } \qquad  \sigma_N \colon {kF^1} \otimes_{kF^0} N &\to N \otimes_{kE^0} {kE^1}
\end{align*}
such that  the diagrams
$$
\xymatrix{
kE^1 \otimes_{kE^0} (M \otimes_{kF^0} N) \ar[rr]^{\sigma_M \# \sigma_N}  \ar[d]_{\operatorname{id} \otimes \omega_E} & & (M \otimes_{kF^0} N) \otimes_{kE^0} kE^1 \ar[d]^{ {\omega_E \otimes \operatorname{id}} }\\
kE^1 \otimes_{kE^0} (kE^1)^{\otimes n} \ar[rr]^{\nu_n^E} & &  (kE^1)^{\otimes n} \otimes_{kE^0} kE^1
}
$$
and
$$
\xymatrix{
kF^1 \otimes_{kF^0} (N \otimes_{kE^0} M) \ar[rr]^{\sigma_N \# \sigma_M}  \ar[d]_{\operatorname{id} \otimes \omega_F} & & (N \otimes_{kE^0} M) \otimes_{kF^0} kF^1 \ar[d]^{ {\omega_F \otimes \operatorname{id}} }\\
kF^1 \otimes_{kF^0} (kF^1)^{\otimes n} \ar[rr]^{\nu_n^F} & &  (kF^1)^{\otimes n} \otimes_{kF^0} kF^1
}
$$
commute.     Then $L_k(E)$ and $L_k(F)$ are graded Morita equivalent.  
\end{thm}

\begin{proof}
We use the $L_k(E) - L_k(F)$ bridging bimodule $Y_{(M, \sigma_M)}$ to construct the functor  
$$\varphi: \mathsf{Mod}\mbox{-}L_k(E) \to \mathsf{Mod}\mbox{-}L_k(F)  \ \ \mbox{ given by }    \ \ \varphi := \underline{\hspace{.25in}}  \otimes_{L_k(E)} Y_{(M, \sigma_M)}  .$$

\noindent
Because $Y_{(M, \sigma_M)}$ is a graded $L_k(E)$--$L_k(F)$-bimodule, Example \ref{tensorfunctorexample} yields that 
$$\varphi :=   \underline{\hspace{.25in}} \otimes_{L_k(E)} Y_{(M, \sigma_M)} :  \mathsf{Mod}\mbox{-}L_k(E) \to  \mathsf{Mod}\mbox{-}L_k(F)\ \   \mbox{ is  a graded functor}.  $$ 
\noindent
The existence of the commutative diagrams of the hypotheses is precisely  what it means for  the specified conjugacy pairs $(M \otimes_{kF^0} N, \sigma_M \# \sigma_N)$ and $((kE^1)^{\otimes n}, \nu_n^E)$ to be equivalent,  and also for the specified conjugacy pairs $(N \otimes_{kE^0} M, \sigma_N \# \sigma_M)$ and $((kF^1)^{\otimes n}, \nu_n^F)$ to be equivalent (see Definition~\ref{equivalent-conj-pairs-def}).  By Proposition~\ref{prop:MUF-iso}, equivalent specified conjugacy pairs have (graded) isomorphic bridging bimodules, and hence
$$ Y_{ ((kE^1)^{\otimes n}, \nu_n^E) } \cong Y_{ (M \otimes_{kF^0} N, \sigma_M \# \sigma_N) }
\qquad \text{ and } \qquad
Y_{((kF^1)^{\otimes n}, \nu_n^F)} \cong Y_{ (N \otimes_{kE^0} M, \sigma_N \# \sigma_M) }
$$
as (graded) $L_k(E)$--$L_k(E)$-bimodules and (graded) $L_k(F)$--$L_k(F)$-bimodules, respectively.  Applying Proposition~\ref{prp-griso-tensor} we conclude
$$ Y_{ ((kE^1)^{\otimes n}, \nu_n^E) } \cong Y_{ (M, \sigma_M) } \otimes_{L_k(F)} Y_{ (N, \sigma_N) }
\quad \text{ and } \quad
Y_{((kF^1)^{\otimes n}, \nu_n^F)} \cong Y_{ (N, \sigma_N) }
\otimes_{L_k(E)}
Y_{ (M, \sigma_M) }
$$
as (graded) $L_k(E)$--$L_k(E)$-bimodules and (graded) $L_k(F)$--$L_k(F)$-bimodules, respectively.  By Lemma~\ref{lem:npowergraph2} we have $Y_{ ((kE^1)^{\otimes n}, \nu_n^E) } \cong L_k(E)$ as  $L_k(E)$--$L_k(E)$-bimodules and that $Y_{((kF^1)^{\otimes n}, \nu_n^F)} \cong L_k(F)$ as  $L_k(F)$--$L_k(F)$-bimodules.  Thus
$$ L_k(E) \cong Y_{ (M, \sigma_M) } \otimes_{L_k(F)} Y_{ (N, \sigma_N) }
\quad \text{ and } \quad
L_k(F) \cong Y_{ (N, \sigma_N) }
\otimes_{L_k(E)}
Y_{ (M, \sigma_M) }
$$
as  $L_k(E)$--$L_k(E)$-bimodules and  $L_k(F)$--$L_k(F)$-bimodules, respectively.   But it is well known  (see e.g. \cite[\S  3.15, Theorem Morita III]{Jacobson}) that  this implies 
$$\varphi:=  \underline{\hspace{.25in}} \otimes_{L_k(E)} Y_{(M, \sigma_M)} : \mathsf{Mod}\mbox{-}L_k(E) \to  \mathsf{Mod}\mbox{-}L_k(F)  \ \   \mbox{ is a category equivalence.} $$

\noindent
The two displayed properties of $\varphi$ thereby yield that   $L_k(E)$ and $L_k(F)$ are graded Morita equivalent (see Definition \ref{gradedequivalentdef} (geMod)), as desired.  
\end{proof}

\subsection{The Relationship with Strong Shift Equivalence}\label{SubsectionSSE}

In this subsection we show that if two finite graphs with no sinks have adjacency matrices that are strong shift equivalent, then the hypotheses of Theorem~\ref{thm:main2} are satisfied, and hence the associated Leavitt path algebras are graded Morita equivalent.  We begin by recalling the definition of strong shift equivalence.

\begin{dfn}
Let $A$ and $B$ be square matrices (of not necessarily the same size) with entries in $\mathbb{Z}_{\geq 0}$.  We say that $A$ and $B$ are \emph{elementary strong shift equivalent} if there exist rectangular matrices $R$ and $S$ such that $A = RS$ and $B = SR$.  We define \emph{strong shift equivalence} to be the transitive closure of this relation; that is, $A$ and $B$ are strong shift equivalent if there exists a sequence $A_1, \ldots A_N$ of square matrices with entries in $\mathbb{Z}_{\geq 0}$ having $A_1 = A$, $A_N = B$, and $A_i$ is elementary strong shift equivalent to $A_{i+1}$ for $1 \leq i \leq N-1$.  Note $A$ and $B$ are shift equivalent with exponent $n=1$ if and only if $A$ and $B$ are elementary strong shift equivalent.  Furthermore, since shift equivalence is an equivalence relation, strong shift equivalence implies shift equivalence.  However, it is a famous result that the converse does not hold; i.e., there are matrices that are shift equivalent but not strong shift equivalent.
\end{dfn}

\begin{prp}[Strong Shift Equivalence implies the Hypotheses of Theorem~\ref{thm:main2}] \label{SE-implies-COM-prop}
Let $E$ and $F$ be finite graphs with no sinks.  If the adjacency matrices of $E$ and $F$ are strong shift equivalent, then the hypotheses of Theorem~\ref{thm:main2} are satisfied.  This implies, in particular, that $L_k(E)$ and $L_k(F)$ are graded Morita equivalent.
\end{prp}

\begin{proof}
It suffices to prove the result when the adjacency matrices are elementary strong shift equivalent, since the general result will follow by composing isomorphisms and concatenating commutative diagrams to show the hypotheses of Theorem~\ref{thm:main2} are satisfied.

Suppose $E$ and $F$ are finite graphs with no sinks and that the adjacency matrices of $E$ and $F$ (which we denote by $A$ and $B$, respectively) are elementary strong shift equivalent.  By definition there exist rectangular matrices $R$ and $S$ such that $A = RS$ and $B = SR$.  As discussed in Definition~\ref{adjacency-matrix-def}, this implies there exist polymorphisms $G$ and $H$ such that $E = G \times H$ and $F = H \times G$.  Applying Proposition~\ref{bimodule-functor-prop} we have bimodule isomorphisms 
$$ kE^1 \cong k G^1 \otimes_{kF^0}  k H^1 \qquad \text{ and } \qquad kF^1 \cong k H^1 \otimes_{kE^0} k G^1.$$
If we set $M := kG^1$ and $N := kH^1$, then we have bimodule  isomorphisms
$$\omega_E \colon M \otimes_{kF^0} N \to kE^1 \qquad \text{ and } \qquad \omega_{F} \colon N \otimes_{kE^0} M \to {kF^1}.$$
Furthermore, we may define bimodule  isomorphisms
$$\sigma_M \colon {kE^1} \otimes_{kE^0} M \to M \otimes_{kF^0} {kF^1} \qquad \text{ and } \qquad  \sigma_N \colon {kF^1} \otimes_{kF^0} N \to N \otimes_{kE^0} {kE^1}$$
by setting 
$$\sigma_M := ( \operatorname{id}_M \otimes \omega_F) \circ \alpha_{M,N,M} \circ (\omega_E^{-1} \otimes \operatorname{id}_M) \ \text{ and } \ \sigma_N := ( \operatorname{id}_N \otimes \omega_E) \circ \alpha_{N,M,N} \circ (\omega_F^{-1} \otimes \operatorname{id}_N).$$\
It is straightforward, although tedious, to verify that
$$\sigma_M \# \sigma_N : kE^1 \otimes_{kE^0} ( M \otimes_{kF^0} N) \to ( M \otimes_{kF^0} N) \otimes_{kE^0} kE^1$$ and $$\sigma_N \# \sigma_M : kF^1 \otimes_{kF^0} ( N \otimes_{kE^0} M) \to ( N \otimes_{kE^0} M) \otimes_{kF^0} kF^1$$ are the unique maps with
$$ \sigma_M \# \sigma_N (x \otimes (m \otimes n)) = \omega_E^{-1}(x) \otimes \omega_E (m \otimes n) \qquad \text{for $x \in kE^1$, $m \in M$, and $n \in N$}$$
and
$$\sigma_N \# \sigma_M (y \otimes (n \otimes m)) = \omega_F^{-1}(y) \otimes \omega_F (n \otimes m) \qquad
\text{for $y \in kF^1$, $n \in N$, and $m \in M$.}$$  
Furthermore, as emphasized in Remark \ref{nu1Eremark}, for $n = 1$ the maps 
$$\nu_1^E : kE^1 \otimes_{kE^0} kE^1 \to kE^1 \otimes_{kE^0} kE^1 \quad \text{  and } \quad \nu_1^F : kF^1 \otimes_{kF^0} kF^1 \to kF^1 \otimes_{kF^0} kF^1$$ are the appropriate identity maps.
Hence the diagrams
$$
\xymatrix{
kE^1 \otimes_{kE^0} (M \otimes_{kF^0} N) \ar[rrr]^{\sigma_M \# \sigma_N = \omega_E^{-1} \otimes \omega_E}  \ar[d]_{\operatorname{id} \otimes \omega_E} &  & & (M \otimes_{kF^0} N) \otimes_{kE^0} kE^1 \ar[d]^{ {\omega_E \otimes \operatorname{id}} }\\
kE^1 \otimes_{kE^0} kE^1 \ar[rrr]^{\nu_1^E = \operatorname{id}} &  & &  kE^1 \otimes_{kE^0} kE^1
}
$$
and
$$
\xymatrix{
kF^1 \otimes_{kF^0} (N \otimes_{kE^0} M) \ar[rrr]^{\sigma_N \# \sigma_M = \omega_F^{-1} \otimes \omega_F}  \ar[d]_{\operatorname{id} \otimes \omega_F} & & & (N \otimes_{kE^0} M) \otimes_{kF^0} kF^1 \ar[d]^{ {\omega_F \otimes \operatorname{id}} }\\
kF^1 \otimes_{kF^0} kF^1 \ar[rrr]^{\nu_1^F = \operatorname{id}} & &  &  (kF^1) \otimes_{kF^0} kF^1
}
$$
trivially  commute, and the hypotheses of Theorem~\ref{thm:main2}  are thereby satisfied.
\end{proof}

\begin{rmk}
It was shown in \cite[Proposition~15(2)]{RoozbehDyn} that if two finite graphs with no sinks and no sources have adjacency matrices that are strong shift equivalent, then the associated Leavitt path algebras are graded Morita equivalent.  The proof given in \cite[Proposition~15(2)]{RoozbehDyn} obtains the result by applying Williams' Theorem,  and then showing that in-splittings and out-splittings of graphs preserve graded Morita equivalence of the associated Leavitt path algebras.   Proposition~\ref{SE-implies-COM-prop} and Theorem~\ref{thm:main2} together establish a slightly more general result than \cite[Proposition~15(2)]{RoozbehDyn} (in that  we do not need the hypothesis  that the graphs have no sources).   Moreover, our proof avoids consideration of in-splittings and out-splittings; indeed, we sidestep these graph techniques entirely, and obtain our result by applying the bridging bimodule.  This provides evidence that the bridging bimodule is a useful tool that not only gives a novel perspective, but allows us to obtain deep results through new techniques.
\end{rmk}

\subsection{The Relationship with the Hazrat Conjecture}\label{SubsectionHazratConj} 

Using results obtained in this article, we may examine the implications for Hazrat's Conjecture and outline new ways of investigating it.

\begin{recast}\label{recast}
Let $E$ and $F$ be finite graphs with no sinks and let $k$ be a field.  Consider the following five statements:
\begin{quote}
\begin{itemize}
\item[{\rm (GrME)}] The Leavitt path algebras $L_k(E)$ and $L_k(F)$ are graded Morita equivalent.
\item[{\rm (GrK)}] There is an order preserving $\mathbb{Z}[x,x^{-1}]$-module isomorphism from $K_0^{\mathrm{gr}}(L_k(E))$ to $K_0^{\mathrm{gr}}(L_k(F))$.
\item[{\rm (SE)}] The adjacency matrices of $E$ and $F$ are shift equivalent.
\item[{\rm (MSE)}] There exists a  $kE^0$--$kF^0$-bimodule $M$, a $kF^0$--$kE^0$-bimodule $N$, and positive integer $n$ such that 
\begin{align*}
{(kE^1)}^{\otimes n} &\cong M \otimes_{kF^0} N & {(kF^1)}^{\otimes n} &\cong N \otimes_{kE^0} M &  \\
{kE^1} \otimes_{kE^0} M &\cong M \otimes_{kF^0} {kF^1} & \text{ and } \qquad  {kF^1} \otimes_{kF^0} N &\cong N \otimes_{kE^0} {kE^1},
\end{align*}
where the isomorphisms are bimodule isomorphisms.
\item[(Com)] There exists a $kE^0$--$kF^0$-bimodule $M$, a $kF^0$--$kE^0$-bimodule $N$, and a positive integer $n$ for which there exist bimodule isomorphisms 
\begin{align*}
\omega_E \colon M \otimes_{kF^0} N &\to ({kE^1})^{\otimes n}, & \omega_{F} \colon N \otimes_{kE^0} M &\to ({kF^1})^{\otimes n} \\
\sigma_M \colon {kE^1} \otimes_{kE^0} M &\to M \otimes_{kF^0} {kF^1}, & \text{ and } \qquad  \sigma_N \colon {kF^1} \otimes_{kF^0} N &\to N \otimes_{kE^0} {kE^1}
\end{align*}
that satisfy the commutation relations
$$\nu_n^E \circ ( \operatorname{id}_{kE^1} \otimes \omega_E) = ( \omega_E \otimes \operatorname{id}_{kE^1}) \circ ( \sigma_M \# \sigma_N) \ \ \ \  \mbox{and} $$
$$\nu_n^F \circ ( \operatorname{id}_{kF^1} \otimes \omega_F) = ( \omega_F \otimes \operatorname{id}_{kF^1}) \circ ( \sigma_N \# \sigma_M).$$

\end{itemize}
\end{quote}
$ $

\noindent The following implications among these five conditions are known to hold:

\medskip

$$
\xymatrix{ {\rm (GrME)} \ar@{=>}[r] & {\rm (GrK)} \ar@{<=>}[r] & {\rm (SE)} \ar@{<=>}[r] & {\rm (MSE)} & (\textnormal{Com}) \ar@{=>}[l] \ar@/_2pc/@{=>}[llll]  }
$$

\noindent As mentioned in the introduction, Hazrat established in \cite{RoozbehDyn} that  ${\rm (GrME)} \implies {\rm (GrK)}$ and ${\rm (GrK)} \iff {\rm (SE)}$.    In this article we have established ${\rm (SE)} \iff {\rm (MSE)}$ (in Theorem~\ref{thm:shift-equivalence-module}),  and $(\textnormal{Com})  \implies {\rm (GrME)}$  (in Theorem~\ref{thm:main2}).  Since  $(\textnormal{Com})$ may be viewed as ${\rm (MSE)}$ plus commutativity conditions of the isomorphisms, we trivially have $(\textnormal{Com}) \implies {\rm (MSE)}$. 
\end{recast}

Hazrat has conjectured that ${\rm (GrK)}$ (equivalently, ${\rm (SE)}$) implies  ${\rm (GrME)}$.  Efforts to establish this conjecture have  been as of yet  unsuccessful.  However, much work continues on this question, since finding necessary and sufficient conditions for graded Morita equivalence, especially easily computable conditions,  remains a high priority in  the study of Leavitt path algebras.

In light of the fact that ${\rm (SE)} \iff {\rm (MSE)}$, the condition stated in $(\textnormal{Com})$ may be considered as ``shift equivalence of the adjacency matrices of $E$ and $F$ together with  certain commutativity conditions".    Given that the community has been unable to show that {\rm (SE)} implies {\rm (GrME)}, the implications above suggest other paths that may be explored.  In particular, we have the following two questions.

\medskip

\noindent \textbf{Question 1:} Does {\rm (GrME)} imply $(\textnormal{Com})$?

\medskip

\noindent \textbf{Question 2:} Does {\rm (MSE)}  imply $(\textnormal{Com})$?  (In other words, can the isomorphisms in {\rm (MSE)} always be chosen to satisfy the two commutativity conditions of $(\textnormal{Com})$?)

\medskip

Given the implications established above, an affirmative answer to Question~2 implies an affirmative answer to Question~1.  If it can be shown that the answer to Question~2 (and hence also Question~1) is ``Yes", then we have equivalence of all five of the above conditions.  

If it can be shown that the answer to Question~1 is ``Yes", while the answer to Question~2 is ``No", then ${\rm (GrME)} \iff (\textnormal{Com})$ and each of these two equivalent conditions is strictly stronger than the equivalent Conditions ${\rm (GrK)}$, ${\rm (SE)}$, and ${\rm (MSE)}$.  

Of course, it is also possible that the answers to Question~1 and Question~2 both turn out to be ``No", in which case $(\textnormal{Com})$ is a strictly stronger condition than {\rm (GrME)}, and hence we need to look at conditions other than $(\textnormal{Com})$ for an equivalent reformulation of {\rm (GrME)}.

Finally, if we let $(\textnormal{SSE})$ denote the property that the adjacency matrices of $E$ and $F$ are strong shift equivalent, then three of our results, displayed together with two results from  \cite{RoozbehDyn}, show the following implications:
$$
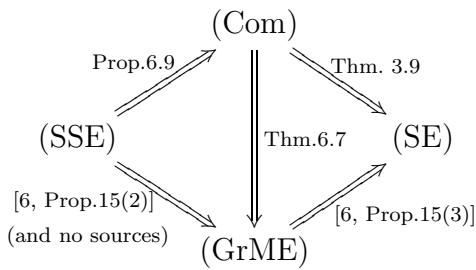
\xymatrix{ 
& (\textnormal{Com}) \ar@{=>}[rd]^{\textnormal{Thm.~\ref{thm:shift-equivalence-module}}} \ar@{=>}[dd]^{\textnormal{Thm.\ref{thm:main2}}}   & \\
(\textnormal{SSE}) \ar@{=>}[ru]^{\textnormal{Prop.\ref{SE-implies-COM-prop}}}  \ar@{=>}[rd]_(.38){\txt{\tiny{\cite[Prop.15(2)]{RoozbehDyn}}  \\ \tiny{
(and no sources)
}}} & & (\textnormal{SE}) \\
& {\rm (GrME)} \ar@{=>}[ru]_{\textnormal{\ \ \cite[Prop.15(3)]{RoozbehDyn}}} & \\
}
$$
which in particular implies  the following chain of implications
$$ (\textnormal{SSE}) \Longrightarrow (\textnormal{Com}) \Longrightarrow {\rm (GrME)} \Longrightarrow (\textnormal{SE}). $$
It is unknown to the authors whether $(\textnormal{Com})$ is equivalent to $(\textnormal{SE})$, equivalent to $(\textnormal{SSE})$, or strictly between the two conditions.  The Hazrat Conjecture asserts  that {\rm (GrME)} is equivalent to $(\textnormal{SE})$.

If (\textnormal{Com}) is equivalent to $(\textnormal{SE})$, then the answers to Question~1 and Question~2 above are both ``Yes", and the Hazrat Conjecture holds.  If (\textnormal{Com}) is strictly between $(\textnormal{SE})$ and $(\textnormal{SSE})$ and the answer to Question~1 is ``Yes", then the Hazrat Conjecture would be false.  If (\textnormal{Com}) is equivalent to $(\textnormal{SSE})$, then the answer to Question~2 above is ``No" with the answer to Question~1 still unknown.  If (\textnormal{Com}) is equivalent to $(\textnormal{SSE})$ and the answer to Question~1 is ``Yes", then the Hazrat Conjecture would be false as {\rm (GrME)} would be equivalent to $(\textnormal{SSE})$.  In contrast to graph $C^*$-algebras, by results of Bratteli and Kishimoto \cite{BK}, and results of Kim and Roush \cite{KimRoush}, there are graph $C^*$-algebras over finite graphs with no sinks that are equivariantly Morita equivalent (the analytic analog of graded Morita equivalence between two Leavitt path algebras) for which the adjacency matrices of the associated graphs are not strong shift equivalent.  If (\textnormal{Com}) is equivalent to $(\textnormal{SSE})$ and the answer to Question~1 is ``No", then the bridging bimodule is rendered to be of limited use in trying to resolve the Hazrat Conjecture.

The fact that $(\textnormal{SE})$ does not imply $(\textnormal{SSE})$ is known as the ``Shift Equivalence Problem" or ``Williams Conjecture", and it is a difficult and tantalizing problem.  The solution to this problem played out over the course of approximately twenty  years (and included the publication of a journal article asserting the equivalency of the two notions that was later found to be incorrect).  A proof that $(\textnormal{SE})$ does not imply $(\textnormal{SSE})$ was ultimately given by Kim and Roush  \cite{KimRoush},  who developed an entirely new invariant (to wit, the {\it sign-gyration-compatibility condition}) to prove that a particular pair of shift equivalent matrices are not strong shift equivalent.

Since it was difficult to show that $(\textnormal{SE})$ does not imply $(\textnormal{SSE})$, and doing so required sophisticated and subtle methods including the development of a new dynamical invariant, this suggests there is a hair's breadth between these two notions.  Thus it may be at least  as difficult to determine where an intermediate notion lies, and so we consequently anticipate that our question of how  (\textnormal{Com}) relates to  $(\textnormal{SE})$ and $(\textnormal{SSE})$ will not be an easy one to resolve.


\end{document}